\documentclass[leqno,12pt]{amsart}
\usepackage[thmnumber=subsection,eqnumber=subsubsection]{jdr-style}
\usepackage{jdr-macros,jdr-tikz}
\usepackage{tikz-cd}
\usepackage{longtable}
\usepackage{url}
\usepackage{verbatim}
\usepackage{tikz}
\usepackage{mathtools}
\usepackage{xr}
\usepackage{enumitem}

\newtheorem{art}[subsection]{}
\newtheorem{artsub}[subsubsection]{}

\title{Classically psh and pluriharmonic functions on Berkovich spaces}

\author[W.~Gubler]{Walter Gubler}
\address{W. Gubler, Mathematik, Universit{\"a}t
	Regensburg, 93040 Regensburg, Germany}
\email{walter.gubler@mathematik.uni-regensburg.de}

\author[J.~Rabinoff]{Joseph Rabinoff}
\address{J. Rabinoff, Department of Mathematics, Trinity College of Arts and Sciences, Duke University, Durham, NC 27708, USA}
\email{jdr@math.duke.edu}

\bibalias{CLD}{chambert_ducros12:forms_courants}

\let\mathbb=\mathbf
\def\artref#1{\ref{#1}}

\tikzcdset{arrow style=math font}

\mathtext{Zar}
\mathtext{Frac}
\mathtext{Int}
\mathtext{PL}
\mathtext{perf}

\mathtext{centdim}
\mathtext{alg}

 \thanks{W.~Gubler
	was supported by the collaborative research
	center SFB 1085 \emph{Higher Invariants - Interactions between Arithmetic Geometry and Global Analysis} funded by the Deutsche Forschungsgemeinschaft.}

\begin{document}

\begin{abstract}
  First we extend the theory of subharmonic functions on smooth strictly $k$-analytic curves from Thuillier's thesis to the case of possibly singular analytic curves over a non-archimedean field $k$. Classically psh functions are then defined as in complex geometry by using pullbacks to analytic curves (and requiring compatibility with base change). We give various properties of classically psh functions including a local and a global maximum principle. As a consequence, we show that the space of pluriharmonic functions on a quasi-compact Berkovich space is finite dimensional.  As a technical tool, we use that a connected Berkovich space is connected by analytic curves.
\end{abstract}


\maketitle


\section{Introduction}

Plurisubharmonic functions are the complex analytic counterparts of subharmonic functions in real analysis. To recall the definition,   we begin with the notion of subharmonicity on the complex line.  Let $\Omega\subset\C$ be an open set and let $u\colon\Omega\to[-\infty,\infty)$ be an upper semicontinuous function.  Then $u$ is \defi{subharmonic} if
for every closed ball $\bar B(a,r)$ with center $a$ and radius $r$ contained in $\Omega$ and every continuous function $h\colon \bar B(a,r) \to \R$ which is harmonic on the open ball $B(a,r)$ such that $u\leq h$ on the boundary $\del\bar B(a,r)$, we have $u\leq h$ on $\bar B(a,r)$.
Note that the notion of (sub)harmonicity is insensitive to biholomorphic maps, hence makes sense on any one-dimensional complex manifold.

Now let $X$ be any complex manifold (without boundary, for simplicity) and let $u\colon X\to[-\infty,\infty)$ be an upper semicontinuous function.
Then $u$ is \defi{plurisubharmonic} or \defi{psh} if, for every holomorphic map $f\colon B(0,1)\to X$ from the open unit ball $B(0,1)\subset\C$, the pullback $u\circ f$ is subharmonic on $B(0,1)$.
We say that $u$ is \defi{pluriharmonic} if $u$ and $-u$ are plurisubharmonic.

Plurisubharmonicity was introduced by Lelong \cite{Lelong42} and Oka \cite{Oka42} independently to study holomorphic convexity. It turned out to be a crucial positivity notion on complex manifolds.
Plurisubharmonic functions satisfy the following properties:
\begin{enumerate}
\item \label{subhar1}
The space of psh functions forms a convex cone: that is, if $u_1,u_2$ are psh and $\lambda_1,\lambda_2\geq 0$, then $\lambda_1u_1 + \lambda_2u_2$ is psh.
\item \label{subhar2}
If $u_1,u_2$ are psh functions then $\max\{u_1,u_2\}$ is psh.
\item \label{subhar3}
If $\{u_k\}_{k\geq1}$ is a decreasing sequence of psh functions, then $u = \lim u_k$ is psh.
\item \label{subhar4}
  A psh function $u$ satisfies the \defi{local maximum principle:} if $u$ attains a local maximum at  $x_0\in X$, then $u$ is constant on the connected component of $x_0$ in $X$.
\end{enumerate}

Now we consider the non-Archimedean situation.  Let $k$ be a field that is complete with respect to a  non-Archimedean
absolute value $|\phantom{a}|$. We refer to Section \ref{section: preliminaries} for the notions used from non-Archimedean geometry. The theory of subharmonic functions on a smooth, strictly analytic $k$-curve $X$ has been completely worked out by Thuillier in his thesis~\cite{thuillier05:thesis}.  Thuillier begins by defining harmonic functions as certain piecewise linear functions satisfying a balancing condition; see~\cite[\S2.3]{thuillier05:thesis} or \artref{sec:thuillier.harmonic}.
He then defines a \defi{subharmonic} function to be an upper semicontinuous function $u\colon X\to\R \cup \{-\infty\}$ such that, for every strictly affinoid domain $U\subset X$ and every harmonic function $h$ on $U$ such that $u\leq h$ on $\del U$, we have $u\leq h$ on $U$.  Thuillier shows that subharmonic functions on smooth non-Archimedean curves satisfy the properties one would expect, in analogy to the complex case.

In higher dimensions, it is not completely clear yet what is the best definition for plurisubharmonic functions. There is the global approach by Zhang~\cite{zhang95}, who defines a semipositive metric on a line bundle $L$ over a projective $k$-variety $X$ to be the uniform limit of metrics induced by nef models. This approach is used in Arakelov geometry for various arithmetic applications. For an ample class  $\theta$ on $X$, Boucksom and Jonsson \cite{BJ22globalpluripotentialtheorytrivially} give a far reaching global pluripotential theory on a projective variety over a trivially valued field. Some parts work for any non-Archimedean field, at least in residue characteristic zero. These psh functions might be singular, and the definition uses decreasing limits of semipositive model functions as in Zhang's approach.
Chambert--Loir and Ducros \cite{chambert_ducros25:forms_courants_v2} have posted a second version of their work in which they give a new definition of psh functions which is functorial and is based on a transfinite regularization process.

In all of the above approaches, a maximum principle for plurisubharmonic functions is missing. In this paper, we propose another class of psh functions called \defi{classically plurisubharmonic functions} in close analogy to the complex case. The  definition is local, the functions are upper semicontinuous,   and we will show that properties \eqref{subhar1}--\eqref{subhar4} are satisfied. It will be clear that classically psh functions form a maximal class of reasonable plurisubharmonic functions that is compatible with Thuillier's definition for curves. In particular, the psh functions from the above are classically psh (see~\ref{sec:psh-func-intro} below).

To get regularization by smooth psh functions in the sense of Chambert-Loir and Ducros or to define a Monge--Amp\`ere measure for psh functions, additional hypotheses as in \cite{chambert_ducros25:forms_courants_v2} might be needed, but that is not the subject of this paper.

\subsection{Subharmonic functions on arbitrary $k$-analytic curves} In Section \ref{section:subharmonic.functions}, we define subharmonic functions on any $k$-analytic curve $X$ to be the elements of the following subsheaf of the sheaf of upper semicontinuous functions on $X$ with range $[-\infty,\infty)$.  It is the smallest subsheaf that is functorial with respect to pullback, stable under base change, and agrees with Thuillier's subharmonic functions on smooth $k$-analytic curves. We will show in Proposition \ref{prop:subharmonic.properties} that subharmonic functions satisfy the properties \eqref{subhar1}--\eqref{subhar3}   as in the complex case, and  the local maximum principle on the  interior $\Int(X)=X \setminus \partial X$.

\subsection{Plurisubharmonic functions on $k$-analytic spaces} \label{sec:psh-func-intro}
We define classically plurisubharmonic functions on $k$-analytic spaces as the elements of the smallest subsheaf of the sheaf of upper semicontinuous functions with range $[-\infty,\infty)$ that is functorial with respect to pullback, stable under base change, and agrees with the above defined sheaf of subharmonic functions on $k$-analytic curves. In Section \ref{section:psh functions}, we will show that classically psh functions satisfy properties \eqref{subhar1}--\eqref{subhar3}.

We use the terminology \emph{classically psh} to distinguish from the psh functions introduced by Chambert-Loir and Ducros in \cite{chambert_ducros25:forms_courants_v2}. Note that  psh functions  are classically psh. Indeed,  psh functions are upper semi-continuous and the definition is functorial, so we are reduced to the case of a smooth analytic curve $C$. We have to see that a psh function is subharmonic. By base change, we may assume that $k$ is algebraically closed and non-trivially valued, see Proposition \ref{prop:subharmonic.properties}(\ref{shp.scalars}). Since subharmonic functions are stable under decreasing limits, and the psh functions are defined by a transfinite regularization process from smooth psh functions, it is enough to show that a smooth psh function on $C$ is subharmonic. This follows from a result of Wanner \cite[Theorem 4.7]{wanner19:subharmonicity}.

\subsection{Plurisubharmonic $\R$-PL functions} For an $\R$-piecewise linear function $h\colon X \to \R$ on a good strictly $k$-analytic space over a non-trivially valued field $k$, we have studied the notion of semipositivity in the paper \cite{GR25semipositivePL}. This notion is a local analytic version of the global semipositive metrics introduced by Zhang \cite{zhang95} which are induced by nef models. We show in Theorem \ref{PL-functions and classically psh} that an $\R$-PL function $h$ is semipositive if and only if $h$ is classically psh.

\subsection{Connectivity by curves}
In Theorem \ref{thm:conn.by.curves}, we show that two points $x,y$ of a connected $k$-analytic space $X$ can be connected by analytic curves (assumed to be compact and defined over a non-Archimedean field extension of $k$). If $\partial X=\emptyset$, then we even show that $x,y$ can be connected by overconvergent curves, which means roughly speaking that we can connect them with the interiors of compact curves. We refer to Section \ref{section:connectivity.by.curves} for the precise definitions. The proof of Theorem \ref{thm:conn.by.curves} is inspired by Berkovich's proof \cite[\S 3.2]{berkovic90:analytic_geometry} showing that a connected $k$-analytic space is pathwise connected. However, we have to choose different paths, as the ones in \emph{ibid} are not analytic curves.  This result is related to~\cite[Proposition~6.1.1]{dejong95:formalrigid} and~\cite[Theorem~4.1.1]{berkovic07:integration} (see the discussion after the statement of Theorem~\ref{thm:conn.by.curves} for more details).

\begin{thm}[Local Maximum Principle]\label{thm:maximum.principle}
  Let $X$ be a  $k$-analytic space, and let $u$ be a classically psh function on $X$.  If $u$ has a local maximum on $x\in\Int(X)$, then $u$ is locally constant at~$x$.
\end{thm}

This will be shown in Theorem~\ref{thm: maximum principle for classically psh functions}. For the proof, one may assume that $\partial X= \emptyset$; then we reduce the claim to the local maximum principle for $k$-analytic curves using connectivity by overconvergent curves.

\subsection{Global maximum principle for affinoids}
Berkovich has shown \cite[Proposition 2.4.4]{berkovic90:analytic_geometry} that a $k$-affinoid space $X=\sM(\sA)$ has a Shilov boundary, i.e.~there is a smallest subset $\Gamma$ of $X$ such that
$$\max_{x \in \Gamma} |f(x)| =\max_{x \in X} |f(x)|$$
for all $f \in \sA$. The Shilov boundary is finite and can be described rather explicitly. Ducros has shown that the relative boundary $\partial X$ is parametrized by analytic spaces defined over non-Archimedean extension fields. An inductive procedure leads us  to a stratification of $\partial X$; we will show in Lemma~\ref{stratification of the boundary'} that the minimal strata precisely make up the Shilov boundary of $X$. Using this fact, we deduce rather easily the following global maximum principle from the local maximum principle (Theorem \ref{thm: global affinoid maximum principle}).

\begin{thm}\label{thm:max.on.shilov.bdy}
  A classically psh  function on a  $k$-affinoid space $X$ attains its maximum on the Shilov boundary of $X$.
\end{thm}

As an application, we will show in Theorem \ref{classically psh and pointwise limits} that a pointwise limit u of a net of classically psh functions is classically psh if $u$ is upper semicontinuous.

\subsection{Pluriharmonic functions}
As in the complex case, a function $u \colon X \to \R$ on the $k$-analytic space $X$ is called \defi{pluriharmonic} if $u$ and $-u$ are plurisubharmonic. We study pluriharmonic functions in Section~\ref{section: pluriharmonic functions}. We will deduce from the above properties of psh functions that pluriharmonic functions form a sheaf of $\R$-vector spaces, they are closed under local uniform limits (even pointwise limits if the limit is a continuous real function), they are stable under base extension and functorial with respect to pullbacks, and they satisfy the local maximum principle. As a consequence of the global maximum principle for $k$-affinoids, we will the following finiteness result in Theorem \ref{thm: finiteness of pluriharmonic}.

\begin{thm}\label{thm:harmonic.finite.dim}
  The set of pluriharmonic  functions on a quasicompact $k$-analytic space is a finite-dimensional real vector space.
\end{thm}

At the end of Section \ref{section: pluriharmonic functions}, we will see that under certain regularity assumptions on the $k$-analytic space $X$, pluriharmonic functions are $\R$-piecewise linear. We do not know if this is true without the regularity assumptions.

\subsection{Acknowledgments}
We thank Michael Termkin for hinting us to the crucial reference for Lemma \ref{lem:base.change.boundaries}.  We also thank Vladimir Berkovich, Mattias Jonsson, and Alex Youcis for comments on an earlier draft of the paper.
We are very grateful to a referee of \cite{GR25semipositivePL} for hinting us to Theorem \ref{classically psh and pointwise limits}.

\section{Preliminaries} \label{section: preliminaries}

In this preliminary section, we gather the foundations used in this paper. While \artref{general notation and conventions}--\artref{extension of scalars} and the notion of the relative boundary in \artref{art:relative.interior} are crucial for the whole paper, the other subsections can be read later when we refer to them.

\subsection{General notation and conventions} \label{general notation and conventions}
In this paper, $k$ will always denote a \defi{non-Archimedean field}, that is,  a field $k$ equipped with a  non-Archimedean complete absolute value $|\phantom{a}|\colon k\to \R$, which may be trivial.

The set of natural numbers $\N$ includes $0$. We allow equality in a set-theoretic inclusion $S \subset T$.
For an  abelian group $P$, let $P_\Lambda \coloneqq P \otimes_\Z \Lambda$.
For a ring $A$, let $A^\times$ be the group of invertible elements.

A topological space is called \defi{compact} if it is quasi-compact and Hausdorff.

\subsection{Ground field and extensions}  \label{Ground field and extensions} \label{art:simultaneous.extensions}
The valuation ring of $k$ is denoted $k^\circ \coloneqq \{\alpha \in k \mid |\alpha| \leq 1\}$, its unique maximal ideal is $k^{\circ\circ}\coloneqq \{\alpha \in k \mid |\alpha| <1 \}$, and its residue field is $\td k \coloneqq k^\circ/k^{\circ\circ}$.
An \defi{analytic extension field} is a non-Archimedean field $k'$ endowed with an isometric embedding $k\inject k'$. For analytic field extensions $k'$ and $k''$ of $k$, there is a joint analytic field extension $F$ of $k'$ and of $k''$: see~\artref{fiber product}.

\subsection{Affinoid algebras} \label{Affinoid algebras}
Consider the polynomial ring $k[T_1,\dots,T_n]$  over~$k$.  For a collection of weights $r_1,\ldots,r_n\in\R_{>0}$, the \defi{weighted Gauss norm} on $k[T_1,\ldots,T_n]$ is given on $f= \sum_{\lambda \in \N^n} a_\lambda T^\lambda$ by
$$\|f\|_r= \max_{\lambda \in \N^n} |a_\lambda| r_1^{\lambda_1} \cdots r_n^{\lambda_n}.$$
We denote by $k\langle r_1^{-1}T_1,\dots, r_n^{-1}T_n \rangle$ the completion of $k[T_1,\dots,T_n]$ with respect to the weighted Gauss norm.
A \defi{$k$-affinoid algebra} is a Banach algebra $\sA$ over $k$ isomorphic as a $k$-algebra to $k\langle r_1^{-1}T_1,\dots, r_n^{-1}T_n \rangle /I$  for some $n \in \N$ and for some (closed) ideal $I$ of $k\langle r_1^{-1}T_1,\dots, r_n^{-1}T_n \rangle$, such that the norm of $\sA$ is equivalent to the residue norm on $k\langle r_1^{-1}T_1,\dots, r_n^{-1}T_n \rangle /I$. If we can choose all $r_i=1$, then $\sA$ is called a \defi{strictly $k$-affinoid algebra}.

\subsection{Affinoid spaces} \label{Affinoid spaces}
We denote by $\sM(\sA)$ the \defi{Berkovich spectrum} of a $k$-affinoid algebra $\sA$. It is the set of bounded multiplicative seminorms $p$ on $\sA$ with $p(1)=1$. We usually denote the points of $X=\sM(\sA)$ by $x$ instead of $p$, and we write $|f(x)|$ instead of $p(f)$ for $f \in \sA$. We endow $X$ with the topology generated by the functions $(|f(\cdot)|)_{f \in \sA}$. Then $X$ is a compact space endowed with a canonical sheaf of rings $\sO_X$, which we call a \defi{$k$-affinoid space}.  If $\sA$ is strictly $k$-affinoid, then $\sM(\sA)$ is called a \defi{strictly $k$-affinoid space}.  See~\cite{berkovic90:analytic_geometry} for details.

The \defi{Shilov boundary} of a $k$-affinoid space $X = \sM(\sA)$ is the unique minimal subset $\Gamma$ of $X$ such that
$$\max_{x \in X} |f(x)| = \max_{x \in \Gamma} |f(x)|$$
for all $f \in \sA$.  A $k$-affinoid space has a finite Shilov boundary by~\cite[Corollary 2.4.5]{berkovic90:analytic_geometry}.

The supremum seminorm $|\phantom{a}|_{\sup}$ on $X=\sM(\sA)$  is used to define the $k^\circ$-algebra $\sA^\circ \coloneqq \{f \in \sA \mid |f|_{\sup}\leq 1\}$, its ideal of topologically nilpotent elements $\sA^{\circ\circ} \coloneqq \{f \in \sA \mid |f|_{\sup}< 1\}$, and the \defi{canonical reduction} $\td\sA\coloneqq \sA^\circ/\sA^{\circ\circ}$.

\subsection{Non-archimedean analytic spaces} \label{non-archimedean analytic spaces}
The $k$-affinoid spaces form the building blocks for $k$-analytic spaces. Roughly speaking, a \defi{$k$-analytic space} $X$ is given by an atlas of compact charts formed by $k$-affinoid spaces. We refer to \cite[\S 1.2]{berkovic93:etale_cohomology} for the precise definition of the category of $k$-analytic spaces and properties. If we can use strictly $k$-affinoid spaces for the charts of $X$, then we call $X$ a \defi{strictly $k$-analytic space.} A $k$-analytic space $X$ is called \defi{good} if every $x \in X$ has a $k$-affinoid neighborhood. These were the spaces considered in \cite{berkovic90:analytic_geometry}.

\subsection{Balls and discs} \label{polydiscs and balls}
We call
$\bB^n \coloneqq \sM(k\langle T_1,\dots, T_n \rangle)$ the \defi{closed unit ball of dimension $n$}.  The \defi{open unit ball} is the open subset $\bB_+^n$ of $\bB^n$ given by
$$\bB_+^n \coloneqq \{x \in X \mid |T_1(x)|<1, \dots, |T_n(x)|<1\}.$$
If $n=1$, then we get the  \defi{closed unit disc} $\bB^1$ and the \defi{open unit disc} $\bB_+^1$. The topology of $\bB^1$ can be described explicitly in terms of closed discs $\bB^1(x,\rho)$ and open discs $\bB_+^1(x,\rho)$ with center $x$ and radius $\rho$, as explained in Appendix~\ref{topology of unit disc}.

\subsection{Grothendieck topology and fibers} \label{Grothendieck topology and fibers}
For a $k$-analytic space $X$, we will also consider the $\G$-topology generated by the analytic domains of $X$, as described in~\cite[\S 1.3]{berkovic93:etale_cohomology}. We denote this space by $X_\G$ to emphasize the use of this Grothendieck topology. For strictly $k$-analytic spaces, we often use the $\G$-topology generated by the strictly analytic domains instead: see~\cite[\S 1.6]{berkovic93:etale_cohomology}.  In either case,
the space $X_\G$ has a well-behaved sheaf of rings $\sO_{X_\G}$. If $V=\sM(\sA)$ is an affinoid domain of $X$, then we have $\sO_{X_\G}(V)=\sA$. For $x \in X$, we denote by $\sH(x)$ the completed residue field of the local ring $\sO_{X_\G,x}$ with respect to the valuation on $\sO_{X_\G,x}/\fm_x$ induced by the norm corresponding to $x$.  If $V$ is any analytic domain in $X$ and if $x \in V$, then $\sH(x)$ is the same whether we consider $x$ as belonging to  $V$ (which is itself a $k$-analytic space), or to $X$.
We call $x \in X$ a \defi{rig-point} if $\sH(x)$ is a finite extension of $k$.

If $\varphi\colon Y \to X$ is a morphism, then the fiber $\varphi^{-1}(x)$ over $x \in X$ is an $\sH(x)$-analytic space, and the underlying topological space has the induced topology \cite[\S 1.4]{berkovic93:etale_cohomology}.

\subsection{Fiber products} \label{fiber product}
The category of $k$-analytic spaces admits fiber products by~\cite[\S 1.4]{berkovic93:etale_cohomology}. For analytic spaces $X,Y$ over a $k$-analytic space $S$, we denote the fiber product by $X \times_S Y$. Given $x \in X$ and $y \in Y$ over the same point $s \in S$, there always exists $z \in X \times_S Y$ mapping to $x$ and to $y$ with respect to the canonical projections~\cite[1.2.13]{ducros18:families}. For convenience, we repeat the argument. The completed tensor product $\sH(x)\hat\otimes_{\sH(s)} \sH(y)$ contains $\sH(x) \otimes_{\sH(s)} \sH(y)$ by a result of Gruson \cite[\S 3.2 Th\'eor\`eme 1(4)]{gruson66}, so $\sH(x)\hat\otimes_{\sH(s)} \sH(y)\neq 0$. The Berkovich spectrum of any nonzero Banach ring is nonempty \cite[Theorem 1.2.1]{berkovic90:analytic_geometry}, and any $z' \in\sM\bigl(\sH(x)\hat\otimes_{\sH(s)} \sH(y)\bigr)$ induces a point $z \in X \times_S Y$ mapping to $x$ and $y$. Note that $\sH(z)$ is a simultaneous analytic extension field of $\sH(x)$ and of $\sH(y)$.

Since Gruson's result holds for Banach algebras over a non-Archimedean field, the same argument shows that for two analytic field extensions of $k$, there is a joint analytic extension of both.

\subsection{Extension of scalars} \label{extension of scalars} \label{art:base.change}
Let $k'/k$ be an an analytic extension field. Then there is a base extension functor $(\cdot)_{k'}$ which associates to a $k$-analytic space $X$ a $k'$-analytic space $X_{k'}$, given on a $k$-affinoid space $X=\sM(\sA)$ by the $k'$-affinoid space $X_{k'}=\sM(\sA \hat{\otimes}_k k')$ \cite[\S 1.4]{berkovic90:analytic_geometry}.  We denote by
$\pi_{k'/k}\colon X_{k'} \to X$
the structure map,  which is closed and surjective by~\cite[Lemma 2.19]{gubler_rabinoff_werner:tropical_skeletons}, \cite[(3.1.1.2)]{ducros14:structur_des_courbes_analytiq}.

\subsection{The Shilov section}\label{art:shilov.section}
For $r=(r_1,\dots,r_n) \in \R_{>0}^n$, we consider the $k$-affinoid domain in the closed polydisc $ \sM\bigl(k\angles{ r_1^{-1}T_1, \dots,r_n^{-1}T_n}\bigr)$ given by the equations $|T_j(x)|=r_j$, $j=1,\dots,n$. The corresponding $k$-affinoid algebra is denoted by $k_r$; it consists of the Laurent series $f=\sum_{\lambda \in \Z^n}a_\lambda T^\lambda$ such that $a_\lambda r^\lambda \to 0$ for $|\lambda|\to \infty$.  The $k$-affinoid algebra $\sM(k_r)$ has a unique Shilov boundary point equal to the weighted Gauss point $\eta_r\in\sM\bigl(k\angles{ r_1^{-1}T_1, \dots,r_n^{-1}T_n}\bigr)$ given by the weighted Gauss norm $\|\phantom a\|_r$ from \artref{Affinoid algebras}.  If $r_1, \dots, r_n$ induce linearly independent elements in the $\Q$-vector space $\R_{> 0}/\sqrt{|k^\times|}$, then $k_r$ is a non-Archimedean field; it follows that $k_r = \sH(\eta_r)$.  For every $k$-affinoid algebra $\sA$, there is always such an $r$ such that $\sA \hat\otimes k_r$ is a strictly $k_r$-affinoid algebra.  See~\cite[pp.~21--22]{berkovic90:analytic_geometry}.

For any $k$-analytic space $X$ and $r  \in \R_{>0}^n$, the structure map $\pi_{k_r/k}\colon X_{k_r} \to X$ has a canonical section $\sigma\colon X \to X_r$ mapping $x \in X$ to the unique Shilov boundary point of the fiber $\pi_{k_r/k}^{-1}(x)=\sM(\sH(x)\hat\otimes k_r) = \sM(\sH(x)_r)$ (which is the weighted Gauss point $\eta_r$). We call $\sigma$ the \defi{Shilov section}.  The Shilov section is continuous: see~\cite[Lemma 3.2.2(i)]{berkovic90:analytic_geometry}, \cite[1.2.16]{ducros18:families}.

\subsection{Relative boundary} \label{relative boundary} \label{art:relative.interior}
An important notion in this paper is that of the \defi{relative boundary} of a morphism $\varphi\colon X \to Y$ of  $k$-analytic spaces. This is a closed subset of $X$ which we denote by $\partial(X/Y)$.
We will not give the details of this definition, but let us try to explain the literature on the subject. In \cite[\S 2.5]{berkovic90:analytic_geometry}, Berkovich defines the relative boundary $\partial(X/Y)$ for $k$-affinoid spaces $X$ and $Y$. Then in \cite[\S 3.1]{berkovic90:analytic_geometry}, he generalizes $\partial(X/Y)$ for good $k$-analytic spaces  using that all points of $X$ and $Y$ have $k$-affinoid neighborhoods. Berkovich calls $\varphi$ \defi{closed} if  $\partial(X/Y)=\emptyset$. Berkovich defines $\del (X/Y)$ for arbitrary $k$-analytic spaces in~\cite[1.5.3(ii)]{berkovic93:etale_cohomology}, as follows.  A morphism $\varphi$ is \defi{closed} if for every morphism $Y'\to Y$ from a good $k$-analytic space $Y'$, the fiber product $X' \coloneqq X\times_Y Y'$ is good and the canonical morphism $X' \to Y'$ is closed in the previous sense. Finally in \cite[Definition 1.5.4]{berkovic93:etale_cohomology}, a point $x \in X$ is defined to be in $\Int(X/Y)$ if $x$ has an open neighborhood $U$ such that $\varphi$ induces a closed morphism $U \to Y$ in the above sense. We refer to \cite[Definition 4.2.4.1]{temkin15:berkovich_intro} for a more explicit equivalent definition of $\partial(X/Y)$. Note that in this paper, we will use the notion \emph{closed map} in the usual topological sense meaning that it maps closed sets to closed sets. This is different from Berkovich's terminology.

The complement $\Int(X/Y)\coloneqq X \setminus \partial(X/Y)$ is called the \defi{relative interior}.  We will usually apply this in the case $Y=\sM(k)$, in which case we call $\partial X \coloneqq \partial(X/\sM(k))$ the \defi{boundary} of $X$ and $\Int(X) \coloneqq X\setminus\del X$ the \defi{interior}. If $\partial X = \emptyset$, we say that $X$ \emph{boundaryless}. A boundaryless space is good by definition. Note that Berkovich calls boundaryless $k$-analytic spaces closed.

We list here some properties of the boundary and interior which we will use throughout the paper:
\begin{enumerate}
\item \label{local on source}
  A point $x\in X$ is contained in $\Int(X/Y)$ if and only if there is an open neighborhood $U$ of $x$ such that $\Int(U/Y) = U$: this is part of the above definition.
\item \label{finite and boundary}
  If $\varphi\colon X \to Y$ is finite, then $\Int(X/Y)=X$ \cite[Example 1.5.3(iii)]{berkovic93:etale_cohomology}.
\item \label{relation to topological boundary}
  If $X$ is an analytic domain in $Y$, then $\Int(X/Y)$ is  the topological interior of $X$ in $Y$ \cite[Proposition 1.5.5(i)]{berkovic93:etale_cohomology}.
\item \label{transitivity of boundaries}
  If $X \stackrel{\varphi}{\rightarrow} Y \stackrel{\psi}{\rightarrow} S$ are morphisms of $k$-analytic spaces, then
  \begin{equation} \label{eq:Int.composition}
    \Int(X/Y) \cap \varphi^{-1}\Int(Y/S) \subset \Int(X/S)
  \end{equation}
  \cite[Proposition 1.5.5(ii)]{berkovic93:etale_cohomology}, with equality if $\psi$ is locally separated \cite[Corollary 5.7]{temkin04:local_properties_II}.  Taking $S = \sM(k)$, this becomes
  \begin{equation}
    \label{eq:Int.morphism}
    \Int(X/Y)\cap\phi\inv\Int(Y)\subset\Int(X).
  \end{equation}
\item \label{boundary and base change}
  If $k'/k$ is an analytic extension field, then
  \[ \pi_{k'/k}^{-1}\Int(X/Y) \subset \Int(X_{k'}/Y_{k'}) \]
  \cite[Proposition 1.5.5(iii)]{berkovic93:etale_cohomology}.  Taking $Y = \sM(k)$, this becomes
  \[ \pi_{k'/k}\inv\Int(X)\subset\Int(X_{k'}). \]
\end{enumerate}
For convenience, we restate these properties in terms of the relative boundary:
\begin{enumerate}[label=(\arabic*$'$)]
\item \label{local on source'}
  A point $x\in X$ is contained in $\del(X/Y)$ if and only if $\del(U/Y)\neq\emptyset$ for every open neighborhood $U$ of $x$.
\item \label{finite and boundary'}
  If $\varphi$ is finite, then $\del(X/Y)=\emptyset$.
\item \label{relation to topological boundary'}
  If $X$ is an analytic domain in $Y$, then $\partial(X/Y)$ is  the topological boundary of $X$ in~$Y$.
\item \label{transitivity of boundaries'}
  If $X \stackrel{\varphi}{\rightarrow} Y \stackrel{\psi}{\rightarrow} S$ are morphisms of $k$-analytic spaces, then
  \begin{equation}\label{eq:del.composition}
    \del(X / S) \subset \del(X / Y)\cup\phi\inv(\del(Y / S)),
  \end{equation}
  with equality if $\psi$ is locally separated.  Taking $S = \sM(k)$, this becomes
  \begin{equation}
    \label{eq:del.morphism}
    \del X \subset \del(X / Y)\cup\phi\inv(\del Y).
\end{equation}
\item \label{boundary and base change'}
  If $k'/k$ is an analytic extension field, then
  \[ \del(X_{k'}/Y_{k'}) \subset \pi_{k'/k}^{-1}\del(X/Y). \]
  Taking $Y = \sM(k)$, this becomes
  \[ \del X_{k'}\subset\pi_{k'/k}\inv\del X. \]
\end{enumerate}
We will need the following lemma, which is a consequence of deep results in~\cite{conrad_temkin21:descent} in general, although it has an elementary proof for good spaces.

\begin{lem}\label{lem:base.change.boundaries}
  If $X$ is a $k$-analytic space and $k'/k$ is an analytic field extension, then
  \[ \pi_{k'/k}(\del X_{k'}) = \del X. \]
\end{lem}

\begin{proof}
  We have $\pi_{k'/k}(\del X_{k'})\subset\del X$ by~\artref{relative boundary}\ref{boundary and base change'}.  Let $x\in\del X$, and let $U\subset X$ be a compact neighborhood of $x$.  We claim that $x\in\del U$.  If not, then there is an open neighborhood $V$ of $x$ in $U$ such that $\del V = \emptyset$ by~\artref{relative boundary}\ref{local on source'}.  If $V' \subset V$ is a smaller open neighborhood of $x$ in $X$, then $\del V' = \emptyset$ as well by~\artref{relative boundary}\ref{relation to topological boundary'} and by~\eqref{eq:del.morphism} as applied to $V' \to V \to \sM(k)$.  Thus we may assume that $V$ is open in $X$, which implies by~\artref{relative boundary}(\ref{local on source})
  that $x\in\Int(X)$, a contradiction.  Hence $x\in\del U$.  If $x'\in\pi_{k'/k}\inv(x)$, then $x'\in\del U_{k'}$ if and only if $x'\in\del X_{k'}$ for the same reason.  We need to show $\pi_{k'/k}\inv(x)\cap\del X_{k'}\neq\emptyset$, so we may replace $X$ by $U$ to assume that $X$ is compact.  In this case, the closed subset $\del X_{k'}$ is compact; suppose that it is disjoint from $\pi_{k'/k}\inv(x)$.  Then $\pi_{k'/k}(\del X_{k'})$ is a closed subset of $X$ not containing $x$, so there exists an open neighborhood $V$ of $x$ such that $V_{k'} = \pi_{k'/k}\inv(V)\subset\Int(X_{k'})$.  This means that $\del V_{k'}=\emptyset$ by~\artref{relative boundary}(\ref{relation to topological boundary},\ref{transitivity of boundaries}).  But $x\in\del V$, and boundarylessness can be checked after extension of scalars by~\cite[Theorem~11.5]{conrad_temkin21:descent}. This contradiction proves the claim.
\end{proof}

\begin{art}\label{rig-smooth}
  A $k$-analytic space $X$ is called \emph{rig-smooth} at $x \in X$ if it is flat (in the sense of \cite{ducros18:families}) and if the sheaf of K\"ahler differentials $\Omega_X$ (considered on the $\G$-topology) is free at $x$ of rank equal to the dimension of $X$ at $y$. We call $X$ \emph{smooth} if it is rig-smooth and boundaryless. We refer to \cite[Chapter 5]{ducros18:families} for a discussion (note that rig-smooth is called quasi-smooth there) and comparison to other definitions.
\end{art}

\subsection{Models and formal geometry} \label{subsection: models and formal geometry}
A quasi-compact \emph{admissible formal scheme}  over $k^\circ$ is a quasi-compact  formal scheme $\fX$ over $k^\circ$ with an open covering by affine formal schemes $\Spf(A)$ for $k^\circ$-algebras $A$ which are topologically of finite type and flat over $k^\circ$. The \emph{generic fiber} $\fX_\eta$ of $\fX$ is the strictly $k$-analytic space obtained by gluing the Berkovich spectra $\sM(A \otimes_{k^\circ} k)$, and the \defi{special fiber}  $\fX_s$ is the scheme over $\td k$ obtained by gluing the affine schemes $\Spec(A \otimes_{k^\circ} \td k)$.

\begin{defn} \label{formal models}
	A \defi{formal $k^\circ$-model} $\fX$ for a quasi-compact strictly $k$-analytic space $X$ is a quasi-compact admissible formal scheme $\fX$ over $k^\circ$ with an identification $\fX_\eta=X$.
\end{defn}

In the above setting, there is a canonical \defi{reduction map} $\red_{\fX}\colon X \to \fX_s$ which is surjective and anticontinuous \cite[Section 2]{gubler_rabinoff_werner16:skeleton_tropical}. We call $x \in X$ a \defi{divisorial point associated to the model $\fX$} if $\red_\fX(x)$ is a generic point of the special fiber $\fX_s$. Note that the set of divisorial points is finite: see \cite[Appendix A]{gubler_martin19:zhangs_metrics}. In {\it ibid}, the valuation on $k$ is assumed to be non-trivial, but this finiteness also holds in the trivially valued case as then $\fX_\eta$ is just the analytification of the special fiber and the generic points of the special fiber have a unique preimage under the reduction.

\subsection{PL functions} \label{subsection: PL functions}
We will assume $k$ is non-trivially valued when we discuss PL functions.
Let $\Lambda$ be a subgroup of $\R$ which is either divisible or equal to $\Z$. We consider a good, strictly $k$-analytic space $X$.
	A function $h\colon X \to \R$ is called \emph{$\Lambda$-piecewise linear} or \emph{$\Lambda$-PL} if $\G$-locally we have $h=\sum_j \lambda_j \log |f_j|$ for finitely many $\lambda_j \in \Lambda$ and invertible analytic functions $f_j$.  Here we use the $\G$-topology of $X$ generated by the strictly $k$-analytic domains. We refer to \cite[Section 5]{gubler_rabinoff_jell:harmonic_trop} for more on $\Lambda$-PL functions.

        \subsection{Reduction of germs} \label{art:reduction.of.germs}
        For a good, strictly $k$-analytic space $X$ over a non-trivially valued field, Temkin \cite{temkin00:local_properties} introduced the reduction of the germ $(X,x)$ at $x \in X$ as an open subset of a Zariski--Riemann space $\P_{\td\sH(x)/\td{k}}$ of $\td\sH(x)/\td{k}$. Roughly speaking, it is the projective limit of all canonical reductions $\Spec(\td{\sA})$ where $\sM(\sA)$ ranges over all $k$-affinoid neighborhoods of $x$. The germ $(X,x)$ is naturally a quasicompact locally ringed space.  See also \cite[Section 6.1]{chambert_ducros12:forms_courants} for an excellent summary.  An $\R$-PL function $h\colon X \to \R$ induces a \defi{residue line bundle} $L(h) \in \Pic(X,x)_\R$.  This was introduced in~\cite{chambert_ducros12:forms_courants}, and is explained in~\cite[Section 6]{gubler_rabinoff_jell:harmonic_trop}.  We call $h$ \defi{semipositive at $x$} if $L(h)$ is nef in the sense of~\cite[Section 7]{gubler_rabinoff_jell:harmonic_trop}. A \defi{semipositive} $\R$-PL function is semipositive at all $x \in X$.

\subsection{Graded reduction} \label{graded reduction}
While the reduction in \ref{art:reduction.of.germs} works well for (good) strictly $k$-analytic spaces, this concept does not fit very well for non-strictly $k$-analytic spaces or in the trivially valued case. To handle the general case, Temkin introduced~\cite{temkin04:local_properties_II} the graded reduction of the germ $(X,x)$. We need it only for Lemma \ref{stratification of the boundary'}, so we introduce here only the \defi{graded reduction} of a $k$-algebra $A$ endowed with a multiplicative seminorm $\rho$. In our applications, the algebra $A$ will be either a $k$-affinoid algebra endowed with $|\phantom{a}|_{\sup}$ or the completed residue field $\sH(x)$ in a point $x$ of a $k$-analytic space endowed with its canonical absolute value. For $r \in \R_{>0}$, we set
$$A^{\leq r} \coloneqq \{a \in A \mid \rho(a) \leq r\}\quad \text{and} \quad A^{<r} \coloneqq \{a \in A \mid \rho(a)<r\}.$$
Then the \defi{graded residue field} of $A$ is given by
$$\td{A}^\bullet \coloneqq \bigoplus_{r \in \R_{>0}} A^{\leq r}/A^{<r}.$$
We refer to \cite{temkin04:local_properties_II}   for properties of this functorial construction.

\section{Subharmonic functions on analytic curves} \label{section:subharmonic.functions}

In this section, we generalize Thuillier's theory \cite{thuillier05:thesis} of subharmonic functions from
smooth strictly analytic curves to arbitrary $k$-analytic curves.

\begin{art}[$k$-Analytic Curves]\label{art:curve.facts}
  By a \defi{$k$-analytic curve} we mean a separated $k$-analytic space of pure dimension $1$.  We gather the following facts about $k$-analytic curves:
  \begin{enumerate}
  \item A $k$-analytic curve is \emph{good}, i.e., every point has an affinoid neighborhood. \cite[Proposition~3.3.7]{ducros14:structur_des_courbes_analytiq}
  \item A $k$-analytic curve is \emph{paracompact}. \cite[Th\'eor\`eme~4.5.10]{ducros14:structur_des_courbes_analytiq}
  \item If $X$ is an affinoid $k$-analytic curve, then $\del X$ coincides with the Shilov boundary of $X$.~\cite[2.1.2]{thuillier05:thesis}
  \end{enumerate}
  An analytic domain in a $k$-analytic curve (for example, an open subset) is again a $k$-analytic curve.
\end{art}

\begin{art}[Thuillier's Harmonic Functions]\label{sec:thuillier.harmonic}
  Thuillier~\cite[\S2.3]{thuillier05:thesis} defines a sheaf of \defi{harmonic} functions on a smooth, strictly $k$-analytic curve for a non-trivially valued $k$, roughly as follows.  If $X$ is affinoid, then a continuous function $h\colon X\to\R$ is harmonic if, after passing to a finite, separable field extension, it coincides with the composition of a harmonic function on a skeleton of $X$ (in the sense of graph theory) with the retraction to the skeleton.  If $X$ is smooth, then a harmonic function on an open subset of $X$ is defined to be a continuous function whose restriction to every strictly affinoid subdomain is harmonic.
\end{art}

\begin{art}[Thuillier's Subharmonic Functions]\label{sec:thuillier.subharmonic}
  For a non-trivally valued $k$, Thuillier~\cite[D\'efinition~3.1.5]{thuillier05:thesis} defines a \defi{subharmonic} function on a smooth, strictly $k$-analytic curve $X$ to be an upper semicontinuous function $u\colon X\to\R\cup\{-\infty\}$ that is not identically $-\infty$ on any connected component of $X$ and that satisfies the following property: for every strictly $k$-affinoid domain $U\subset X$ and every harmonic function $h\colon U\to\R$, if $u\leq h$ on $\del U$ then $u\leq h$ on $U$.
\end{art}

  Now we extend this definition to the case of non-smooth curves.  The idea is that a function should be subharmonic if it pulls back to a subharmonic function along any morphism from a smooth curve.  However, an inseparable curve over an imperfect field does not admit any non-constant morphisms from a smooth curve, so one has to allow an extension of scalars as well.  We also consider the constant function $-\infty$ to be subharmonic, following Demailly~\cite{demailly85} in the complex case.

\begin{defn}[Subharmonic functions]\label{def:subharmonicity}
  Let $X$ be a $k$-analytic curve and let $u\colon X\to\R\cup\{-\infty\}$ be an upper semicontinuous function.  We say that $u$ is \defi{subharmonic} provided that, for every non-trivially valued analytic extension field $k'/k$, every smooth, connected, strictly $k'$-analytic curve $Y$, and every morphism $Y\to X_{k'}$, the composition
  \[ Y\To X_{k'} \overset{\pi_{k'/k}}\To X \overset{u}\To \R\cup\{-\infty\} \]
  is either subharmonic on $Y$ in the sense of~\cite[D\'efinition~3.1.5]{thuillier05:thesis}, or is identically~$-\infty$.
\end{defn}

\begin{artsub}\label{art:recover.thuillier.subharmonicity}
  If $X$ is a smooth, connected, strictly $k$-analytic curve over a non-trivially valued $k$, then an upper semicontinuous function $u\colon X\to\R\cup\{-\infty\}$ is subharmonic in our sense if and only if it is either subharmonic in Thuillier's sense or is identically $-\infty$.  Indeed, if $u$ is subharmonic in our sense then one can take $k=k'$ and $Y=X$ to see that it is either subharmonic in Thuillier's sense or is identically $-\infty$.  Conversely, if $u$ is identically $-\infty$ then it is clearly subharmonic in our sense, and if $u$ is subharmonic in Thuillier's sense then it is subharmonic in our sense by~\cite[Proposition~3.1.14 and Corollaire~3.4.5]{thuillier05:thesis}.
\end{artsub}

\begin{artsub}\label{art:subharmonicity.real.values}
  As a special case of~(\ref{art:recover.thuillier.subharmonicity}), if $X$ is a smooth, connected, strictly $k$-analytic curve over a non-trivially valued $k$, then a function $u\colon X\to\R$ (taking finite values) is subharmonic in our sense if and only if it is subharmonic in Thuillier's sense.
\end{artsub}

Here we collect the basic properties of subharmonic functions, which we will prove over the course of this section.

\begin{prop}\label{prop:subharmonic.properties}
  Let $X$ be a $k$-analytic curve and let $u\colon X\to\R\cup\{-\infty\}$ be an upper semicontinuous function.
  \begin{enumerate}
  \item\label{shp.sheafiness} \textup{(Sheafiness)}~ The subharmonic functions form a sheaf on $X$.
  \item\label{shp.analytic} \textup{(Analytic functions)}~ If $f\in\Gamma(X,\sO_X)$ then $\log|f| : X\to\R\cup\{-\infty\}$ is subharmonic.
  \item\label{shp.boundaries} \textup{(Boundaries)}~ If $u$ is subharmonic on $\Int(X)$, then $u$ is subharmonic on $X$.
  \item\label{shp.cone} \textup{(Cone Property)}~ If $u_1,u_2\colon X\to\R\cup\{-\infty\}$ are subharmonic and $\lambda_1,\lambda_2\in\R_{\geq 0}$ then $\lambda_1 u_1+ \lambda_2 u_2$ and $\max\{u_1,u_2\}$ are subharmonic.
  \item\label{shp.limits} \textup{(Limits)}~ The infimum of a decreasing net of subharmonic functions is subharmonic.
  \item\label{shp.maximum} \textup{(Maximum Principle)}~ If $u$ is subharmonic and $u$ attains a local maximum at $x\in\Int(X)$, then $u$ is constant in a neighborhood of $x$.
  \item\label{shp.scalars} \textup{(Extension of scalars)}~ Let $k'/k$ be an analytic extension field.  Then $u$ is subharmonic if and only if $u\circ\pi_{k'/k}\colon X_{k'}\to\R\cup\{\infty\}$ is subharmonic.
  \item\label{shp.functoriality} \textup{(Functoriality)}~ Let $Y$ be a $k$-analytic curve and let $f\colon Y\to X$ be a morphism.  If $u$ is subharmonic then $u\circ f$ is subharmonic, and the converse holds if $f$ is finite and surjective.
  \item\label{shp.-infty}  \textup{(Finite Values)}~ If $u$ is subharmonic and $u(x)=-\infty$ for a non-rig-point $x\in X$, then $u\equiv-\infty$ on the irreducible component of $\Int(X)\cup\{x\}$ containing $x$.
  \end{enumerate}
\end{prop}

Let us make some remarks regarding the content of Proposition~\ref{prop:subharmonic.properties}.

\begin{artsub}
  If $U\subset X$ is an open subset then $U$ is again a $k$-analytic curve, so it makes sense to speak of subharmonic functions on $U$.  It is immediate from the definition that a subharmonic function on $X$ restricts to a subharmonic function on $U$.  Hence subharmonic functions form a \emph{presheaf} on $X$.  Proposition~\ref{prop:subharmonic.properties}(\ref{shp.sheafiness}) asserts that this presheaf is a sheaf.
\end{artsub}

\begin{artsub}
  The maximum principle (Proposition~\ref{prop:subharmonic.properties}(\ref{shp.maximum})) is~\cite[Proposition~3.4.10]{thuillier05:thesis} in the smooth, strictly analytic case.
\end{artsub}

\begin{artsub}\label{art:basechange.difficult}
  In the situation of Proposition~\ref{prop:subharmonic.properties}(\ref{shp.scalars}), it is immediate from Definition~\ref{def:subharmonicity} that $u\circ\pi_{k'/k}$ is subharmonic when $u$ is subharmonic.  This is the \emph{difficult} direction in the smooth, strictly analytic case: see~\cite[Corollaire~3.4.5]{thuillier05:thesis}.  This result is hidden in the fact that Definition~\ref{def:subharmonicity} essentially recovers Thuillier's definition of subharmonicity in that case~\artref{art:recover.thuillier.subharmonicity}.
\end{artsub}

\begin{artsub}
  In the situation of Proposition~\ref{prop:subharmonic.properties}(\ref{shp.functoriality}), it is immediate from Definition~\ref{def:subharmonicity} that $u\circ f$ is subharmonic when $u$ is subharmonic.  The remark in~\artref{art:basechange.difficult} applies to a lesser extent in this case: see~\cite[Proposition~3.1.14]{thuillier05:thesis}.
\end{artsub}

\begin{artsub} \label{unique irreducible}
  For any $x\in X$, the subset $\Int(X)\cup\{x\}$ is \emph{open} in $X$.  This is obvious if $x\in\Int(X)$, and if $x\in\del X$ then it follows from the fact that $\del X$ is \emph{discrete} in $X$.  In particular, $\Int(X)\cup\{x\}$ is again a $k$-analytic curve, so it makes sense to speak of its irreducible components in Proposition~\ref{prop:subharmonic.properties}(\ref{shp.-infty}).  Irreducible components can only meet at rig-points (their intersections are Zariski-closed), so a non-rig-point $x$ is contained in a unique irreducible component of $\Int(X)\cup\{x\}$; this is implicit in Proposition~\ref{prop:subharmonic.properties}(\ref{shp.-infty}).
\end{artsub}

\begin{artsub} \label{unique irreducible implies neighborhood}
  Proposition~\ref{prop:subharmonic.properties}(\ref{shp.-infty}) implies that if $u$ is subharmonic and $u(x)=-\infty$ for a non-rig-point $x\in X$, then $u\equiv-\infty$ in a neighborhood of $x$.
\end{artsub}

\begin{art} \label{three steps of reduction}
  In order to extend Thuillier's results, we will pass from a $k$-analytic curve to a
smooth, strictly $k$-analytic curve in four steps:
\begin{enumerate}
\item Extend scalars to a field as in~\artref{art:shilov.section} to make the curve strictly analytic.
\item If $\chr(k)=p$, then extend scalars to the completion of the perfect closure $k^{p^{-\infty}}$.
\item Replace the analytic curve by the underlying reduced curve.
\item Pass to the interior of the normalization, which is smooth (see~\artref{rig-smooth}).
\end{enumerate}
We begin with some lemmas about this process.
\end{art}

\begin{lem}\label{lem:reduce.to.perfect}
  Suppose that $\chr(k)=p > 0$.  Let $k'$ be the completion of $k^{p^{-\infty}}$.  Then $k'$ is perfect, and  if $X$ is a $k$-analytic space then $\pi_{k'/k}\colon X_{k'}\to X$ is a homeomorphism such that $\pi(\del X_{k'}) = \del X$.
\end{lem}

\begin{proof}
  We leave it as an exercise for the reader to verify that $k'$ is perfect.  The fact that $\pi_{k'/k}$ is a homeomorphism is~\cite[Remarque~0.5]{ducros09:excellent}, and the assertion about boundaries is Lemma~\ref{lem:base.change.boundaries}.
\end{proof}

We make the following observation to mirror Lemma~\ref{lem:reduce.to.perfect}.  Its proof is immediate.

\begin{lem}\label{lem:reduce.to.reduced}
  Let $X$ be a strictly $k$-analytic space and let $\iota\colon X_{\red}\inject X$ be the inclusion of the underlying reduced space.  Then $\iota$ is a homeomorphism of underlying topological spaces, and $\iota(\del X_{\red}) = \del X$.
\end{lem}

We refer the reader to~\cite[before Proposition~3.1.8]{berkovic90:analytic_geometry} for the construction of the normalization.

\begin{lem}\label{lem:reduce.to.smooth}
  Let $X$ be a reduced, strictly $k$-analytic curve and let $\nu\colon\td X\to X$ be its normalization.  Then the following hold:
  \begin{enumerate}
  \item $\td X$ is rig-smooth if $k$ is perfect.
  \item $\nu$ is an isomorphism outside of a discrete set of rig-points of $X$.
  \item $\nu$ induces a bijection $\del\td X\isom\del X$.
  \item $\nu$ is a topological quotient.
  \item Let $x\in X$, let $U$ be an open neighborhood of $x$, and consider the map $\td U = \nu\inv(U)\to U$.  If $U$ is a sufficiently small connected open neighborhood of $x$, then:
    \begin{enumerate}
    \item $\td U\to U$ is injective on connected components of $\td U$.
    \item $\td U\to U$ restricts to an isomorphism $\td U\setminus\nu\inv(x)\isom U\setminus\{x\}$.
    \item Every connected component of $\td U$ intersects $\nu\inv(x)$.
    \end{enumerate}
  \item A function $h\colon X\to\R$ is $\R$-PL if and only if $h\circ\nu$ is $\R$-PL.
  \end{enumerate}
\end{lem}

\begin{proof}
  The normalization $\td X$ is locally the Berkovich spectrum of a Dedekind domain, so it is rig-smooth  when $k$ is perfect because it is regular; this proves~(1).  Normalization is an isomorphism away from codimension~$1$ as in algebraic geometry, which yields~(2).  Since $\nu$ is finite, we have $\del\td X = \nu\inv(\del X)$ by the properties \ref{finite and boundary'} and \ref{transitivity of boundaries'} of~\artref{relative boundary}, so~(3) follows from~(2) because the boundaries do not contain any rig-points.  Assertion~(4) was noted in~\cite[(3.1.3)]{ducros14:structur_des_courbes_analytiq}.

  Now we prove~(5).  Since normalization is a local construction, we may shrink $X$ to assume that it is normal away from $x$.  Suppose that $\nu\inv(x) = \{x_1,\ldots,x_n\}$, and choose disjoint, connected open neighborhoods $U'_1,\ldots,U'_n$ of $x_1,\ldots,x_n$, respectively.  Let $U' = \bigcup_{i=1}^n\nu(U'_i)$.  Since $\nu$ induces an isomorphism $\td X\setminus\{x_1,\ldots,x_n\}\isom X\setminus\{x\}$, it is clear that $\nu\inv(U') = \Djunion_{i=1}^nU'_i$. Since $\nu$ is a topological quotient, we have that $U'$ is open, and $\nu$ is injective on each $U'_i$ by construction.

  Suppose now that $U\subset U'$ is a connected open neighborhood of $x$. By \cite[Theorem 3.2.1]{berkovic90:analytic_geometry}, $U$ is path-connected.  Clearly $U$ satisfies~(a) and~(b).  Let $U_i = \nu\inv(U)\cap U'_i$, so that $x_i\in U_i$ and $\td U = \nu\inv(U) = \Djunion_{i=1}^nU_i$.  In order to prove~(c), it is enough to show that $U_i$ is path-connected.  Let $y\in U_i\setminus\{x_i\}$, and let $\gamma\colon[0,1]\to U$ be a path from $\pi(y)$ to $x$.  We may assume that $\gamma\inv(x) = \{1\}$, so that $\gamma$ restricts to a path $[0,1)\to U\setminus\{x\}$.  Let $\td\gamma\colon[0,1)\to\td U\setminus\nu\inv(x)$ be the composition of $\gamma|_{[0,1)}$ with the inverse of the isomorphism $\td U\setminus\nu\inv(x)\isom U\setminus\{x\}$.  Then $\td\gamma(0) = y$, so the image of $\td\gamma$ is contained in $U_i$.  Extend $\td\gamma$ to a function $[0,1]\to U_i$ by setting $\td\gamma(1) = x_i$.  One checks continuity of $\td\gamma$ at $1$ using the fact that an open neighborhood of $x$ is the same as a collection of open neighborhoods of $x_1,\ldots,x_n$, as in the previous paragraph.  This shows that $U_i$ is path-connected.

  In the situation of~(6), if $h$ is $\R$-PL then $h\circ\nu$ is $\R$-PL by~\cite[Lemma~5.4(1)]{gubler_rabinoff_jell:harmonic_trop}.  Suppose then that $h\circ\nu$ is $\R$-PL.  Piecewise linearity of $h$ is local on $X$, and $\nu$ is an isomorphism away from a finite set of (non-normal) rig-points, so it is enough to check that $h$ is $\R$-PL in a neighborhood of a single rig-point~$x$.  Let $x'\in\td X$ be a preimage of $x$.  There is a covering of $\td X$ by strictly affinoid domains on which $h\circ\nu$ has the form $\sum_{i=1}^n\lambda_i\log|f_i|$, where $\lambda_i\in\R$ and $f_i\in\sO^\times$.  Since $x'$ is a rig-point, it is contained in the \emph{interior} of one of these affinoid domains, so $h\circ\nu = \sum_{i=1}^n\lambda_i\log|f_i|$ in a neighborhood of $x'$.  But if $f$ is invertible in a neighborhood of a rig-point $x'$, then $|f|$ is constant on a neighborhood of $x'$ (this elementary fact can be seen as a very special case of \cite[Theorem 3.4]{ducros12:squelettes_modeles}).  It follows that $h\circ\nu$ is constant in a neighborhood of $x'$.  Suppose that $\nu\inv(x) = \{x_1,\ldots,x_n\}$.  For each $i$ choose an open neighborhood $U_i$ of $x_i$ on which $h\circ\nu$ is constant.  As above, $U = \bigcup_{i=1}^n\nu(U_i)$ is an open neighborhood of $x$, and clearly $h$ is constant, hence $\R$-PL, on~$U$.
\end{proof}

\begin{proof}[Proof of Proposition~\ref{prop:subharmonic.properties}(\ref{shp.sheafiness})]
  This follows from the corresponding fact for subharmonic functions on smooth, strictly analytic curves~\cite[Corollaire~3.1.13]{thuillier05:thesis}, except that one has to be careful about where a function takes the value $-\infty$.  Suppose then that $\{U_i\}$ is an open cover of $X$ and that $u\colon X\to\R\cup\{-\infty\}$ is subharmonic on each $U_i$.  Note that $u$ is upper semicontinuous.  Let $k'/k$ be an analytic extension field with $k'$ non-trivially valued, let $Y$ be a smooth, connected, strictly $k'$-analytic curve, and let $f\colon Y\to X_{k'}$ be a morphism.  Assume that $u\circ\pi_{k'/k}\circ f\not\equiv-\infty$, so we must show that $u\circ \pi_{k'/k}\circ f$ is subharmonic in Thuillier's sense.  Let $U_i' = \pi_{k'/k}\inv(U_i) = (U_i)_{k'}$ and let $Y_i = f\inv(U_i')$.  Since $u$ is subharmonic on $U_i$, the composition $u\circ\pi_{k'/k}\circ f|_{Y_i}$ is either subharmonic in Thuillier's sense or is identically $-\infty$ on each connected component of $Y_i$.

  Let $W_1$ be the set of all points $y\in Y$ that admit a neighborhood on which $u\circ\pi_{k'/k}\circ f$ is subharmonic in Thuillier's sense, and let $W_2$ be the set of all points $y\in Y$ that admit a neighborhood on which $u\circ\pi_{k'/k}\circ f\equiv-\infty$.  Clearly $W_1$ and $W_2$ are open, and they cover $Y$ since $Y$ is covered by the connected components of the $Y_i$.  They are disjoint because if $u\circ\pi_{k'/k}\circ f$ is subharmonic in Thuillier's sense on a connected neighborhood of $y\in Y$, then it cannot be identically $-\infty$ on any neighborhood of $y$~\cite[Lemme~3.1.9]{thuillier05:thesis}.  Since $Y$ is connected and $u\circ\pi_{k'/k}\circ f\not\equiv-\infty$, we have $Y = W_1$.  The result now follows from~\cite[Corollaire~3.1.13]{thuillier05:thesis}.
\end{proof}

\begin{proof}[Proof of Proposition~\ref{prop:subharmonic.properties}(\ref{shp.analytic})]
  Let $f$ be an analytic function on $X$. The function $|f|$ is continuous and hence $u \coloneqq \log|f|$ is upper semicontinuous as  a continuous function with values in $\R \cup \{-\infty\}$. According to our definition of subharmonic functions, we have to show that after base change of $u$ with respect  to an analytic field extension $k'/k$  with $k'$ non-trivially valued and then pull-back to a smooth strictly $k'$-analytic curve, the resulting function is subharmonic. This boils down to showing that $u$ is subharmonic in the special case of a smooth strictly $k$-analytic curve $X$. For this, we may assume $X$ connected. If $f=0$, then $u$ is identically $-\infty$ and hence subharmonic in the sense of Definition \ref{def:subharmonicity}. If $f$ is not identically zero, then the Poincar\'e--Lelong formula \cite[Proposition 3.3.15]{thuillier05:thesis} and the characterization in terms of positive currents \cite[Proposition 3.4.4]{thuillier05:thesis} show that $u$ is subharmonic in the sense of Thuillier.
\end{proof}

We will use the following lemma in the proofs of Proposition~\ref{prop:subharmonic.properties}(\ref{shp.boundaries},\ref{shp.-infty}).  If $X$ is a $k$-analytic space and $k'/k$ is an analytic extension field, then  $\pi_{k'/k}\inv(x) = \sM(\sH(x)\hat\tensor_kk')$ has a finite Shilov boundary for any Abhyankar point $x\in X$ by~\cite[3.2.14]{ducros14:structur_des_courbes_analytiq}; we denote this boundary by $x_{[k']}$.

\begin{lem}\label{lem:fiber.boundary}
  Let $X$ be a $k$-analytic curve and let $k'/k$ be an analytic field extension.  Then for all $x\in\del X$ we have $(\del X_{k'})\cap\pi_{k'/k}\inv(x) = x_{[k']}$.
\end{lem}

\begin{proof}
  Replacing $X$ by an affinoid neighborhood of $x$, we may assume that $X = \sM(\sA)$ is affinoid.  By a result of Ducros~\cite[Lemme 3.1]{ducros12:squelettes_modeles}, there exists a morphism $g\colon X\to\bB^1(0,s)$ for some $s\in\R_{>0}$ such that $x\in g\inv(\eta_s)$, where $\eta_s$ is the maximal point of $\bB^1(0,s)$.  This is discussed in detail in Section~\ref{section: maximum principle for classically functions}.  Consider the commutative square
  \[
    \begin{tikzcd}
      X_{k'} \rar["g'"] \dar["\pi_{k'/k}"'] & \bB^1_{k'}(0,s) \dar["\pi_{k'/k}"] \\
      X \rar["g"'] & \bB^1(0,s)\rlap.
    \end{tikzcd}
  \]
  Since $\del\bB^1_{k'}(0,s)$ contains (only) its maximal point $\eta_{s,k'}$, we have $(g')\inv(\eta_{s,k'})\subset\del X_{k'}$ by~\eqref{eq:del.morphism} (which is an equality because $\bB^1_{k'}(0,s)$ is separated).  The morphism~$g$ is an \textit{Abhyankar presentation} of $x$ in the sense of~\cite[3.2.12.2]{ducros14:structur_des_courbes_analytiq}, so by Proposition~3.2.13 of \textit{ibid}, we have $x_{[k']} = (g')\inv(\eta_{s,k'})\cap\pi_{k'/k}\inv(x)\subset(\del X_{k'})\cap\pi_{k'/k}\inv(x)$.

  For the reverse inclusion, let $x'\in(\del X_{k'})\cap\pi_{k'/k}\inv(x)$, and suppose that $x'\notin x_{[k']}$.  Since $\del X_{k'}$ is the Shilov boundary of $X_{k'} = \sM(\sA\hat\tensor_kk')$, by~\cite[Corollary~2.4.5]{berkovic90:analytic_geometry} there exist $f\in\sA\hat\tensor_k k'$ and $\epsilon>0$ such that $|f|$ attains its maximum value at $x'$ and such that
  \[ \bigl\{y\in X_{k'}\bigm||f(y)|>|f(x')|-\epsilon\bigr\}\subset X_{k'}\setminus x_{[k']}. \]
  In particular, the restriction of $f$ to $\pi_{k'/k}\inv(x)$ does not attain its maximum value on $x_{[k']}$, which cannot happen.  Thus $x'\in x_{[k']}$, so $(\del X_{k'})\cap\pi_{k'/k}\inv(x)\subset x_{[k']}$.
\end{proof}

\begin{proof}[Proof of Proposition~\ref{prop:subharmonic.properties}(\ref{shp.boundaries})]
  Suppose that $u$ is subharmonic on $\Int(X)$.  Let $k'/k$ be an analytic extension field  with $k'$ non-trivially valued, let $Y$ be a smooth, connected, strictly $k'$-analytic curve, and let $f\colon Y\to X_{k'}$ be a morphism.  We must prove that $u\circ\pi_{k'/k}\circ f$ is subharmonic in Thuillier's sense or is identically $-\infty$.  Since $\del Y = \emptyset$, it follows from~\artref{relative boundary}(\ref{transitivity of boundaries})
  that $f(Y)\subset\Int(X_{k'})$.  If $\pi_{k'/k}\circ f(Y)\subset\Int(X)$ then $u\circ\pi_{k'/k}\circ f$ is subharmonic in Thuillier's sense or is identically $-\infty$ because $u$ is subharmonic on $\Int(X)$.  Suppose then that $\pi_{k'/k}\circ f(y) = x\in\del X$ for some $y\in Y$, so that $f(y)\in\Int(X_{k'})\cap\pi_{k'/k}\inv(x)$.

  We claim that $\pi_{k'/k}\inv(x)\cap\Int(X_{k'})$ is both open and closed in $\Int(X_{k'})$.   By~\cite[Proposition~5.3.4]{ducros14:structur_des_courbes_analytiq}, if $U$ is a connected component of $\pi_{k'/k}\inv(X\setminus\{x\})$, then the closure of $U$ in $X_{k'}$ is contained in $\pi_{k'/k}\inv(X\setminus\{x\})\cup x_{[k']}$.
  Let $U'$ be the connected component of $X_{k'}\setminus x_{[k']}$ containing $U$. The above shows that $U$ is closed in $U'$.   Since $U$ is open in $X$ (an analytic space is locally connected~\cite[Remark 1.2.4(iii)]{berkovic93:etale_cohomology} and hence its connected components are open~\cite[\href{https://stacks.math.columbia.edu/tag/02KB}{Lemma 04ME}]{stacks-project}), it is open in $U'$, so $U=U'$ since $U'$ is connected.
  Hence $\pi_{k'/k}\inv(X\setminus\{x\})$ is a union of connected components of $X_{k'}\setminus x_{[k']}$, so the same is true of $\pi_{k'/k}\inv(x)\setminus x_{[k']}$.
   Since $X_{k'}\setminus x_{[k']}$ is itself a $k$-analytic space as $x_{[k']}$ is a finite set, its connected components are open, so we conclude that
    $\pi_{k'/k}\inv(x)\setminus x_{[k']}$ is open and closed in $X_{k'}\setminus x_{[k']}$.
    Since $x \in \partial X$, we have  $(\pi_{k'/k}\inv(x)) \cap \partial  X_{k'}= x_{[k']}$ by Lemma~\ref{lem:fiber.boundary}, which shows
    $\Int(X_{k'})\subset X_{k'}\setminus x_{[k']}$ and $\pi_{k'/k}\inv(x)\setminus x_{[k']} = \pi_{k'/k}\inv(x)\cap\Int(X_{k'})$.
    We conclude that $\pi_{k'/k}\inv(x)\setminus x_{[k']}$ is both open and closed in $\Int(X_{k'})$, as claimed.

    Since $Y$ is connected, $f(Y) \subset \Int(X_{k'})$, and $f(y)\in\Int(X_{k'})\cap\pi_{k'/k}\inv(x)$, we deduce from the above claim that $f(Y)\subset\Int(X_{k'})\cap\pi_{k'/k}\inv(x)$.  Hence $\pi_{k'/k}\circ f$ is constant, so $u\circ\pi_{k'/k}\circ f$ is constant.  It follows that $u$ is subharmonic in Thuillier's sense or is identically $-\infty$.
\end{proof}

\begin{proof}[Proof of Proposition~\ref{prop:subharmonic.properties}(\ref{shp.cone},\ref{shp.limits})]
  These follow quickly from the corresponding facts for subharmonic functions on smooth, strictly analytic curves~\cite[Proposition~3.1.8]{thuillier05:thesis}.
\end{proof}

\begin{proof}[Proof of Proposition~\ref{prop:subharmonic.properties}(\ref{shp.maximum})]
  Suppose first that $X$ is strictly $k$-analytic for $k$ non-trivially valued.  If $\chr(k) = p > 0$ then let $k'$ be the completion of $k^{p^{-\infty}}$; otherwise let $k'=k$.  Let $X' = X_{k'}$ and let $\pi = \pi_{k'/k}\colon X'\to X$, so $\pi$ is a homeomorphism preserving boundaries by Lemma~\ref{lem:reduce.to.perfect}.  Let $\iota\colon X'_{\red}\inject X'$ be the inclusion of the underlying reduced space and let $\nu\colon\td X'\to X'_{\red}$ be the normalization.  Note that $\nu$ and $\iota$ are finite, so $\del\td X' = (\iota\circ\nu)\inv(\del X')$, and hence $\del\td X' = (\pi\circ\iota\circ\nu)\inv(\del X)$.  Let $Y = \Int(\td X')$ and let $f = \iota\circ\nu|_Y\colon Y\to X'$.  Then $Y$ is smooth and strictly analytic, so $u\circ\pi\circ f$ is by definition either subharmonic in Thuillier's sense or identically $-\infty$ on every connected component of $Y$.

  Suppose that $u$ attains a local maximum at $x\in\Int(X)$.  If $u(x)=-\infty$ then $u$ must be constant in a neighborhood of $x$ by semicontinuity of $u$.  Assume then that $u(x)>-\infty$.   Both $\pi$ and $\iota$ are homeomorphisms, so it suffices to show that $u\circ\pi\circ\iota\colon X'_{\red}\to\R\cup\{-\infty\}$ is constant in a neighborhood of the unique point $x'\in\Int(X'_{\red})$ mapping to $x$.  Choose a connected open neighborhood $U\subset\Int(X'_{\red})$ of~$x'$ as in Lemma~\ref{lem:reduce.to.smooth}(5), on which $u\circ\pi\circ\iota$ attains a maximum at $x'$.
  By Lemma~\ref{lem:reduce.to.smooth}(3), we have $\nu^{-1}(U)\subset Y=\Int(\td X')$, so $u\circ\pi\circ f$ is again either subharmonic in Thuillier's sense or identically $-\infty$ on each connected component $C$ of $\nu^{-1}(U)$; since $u(x)>-\infty$ and $C\cap\nu\inv(x')\neq\emptyset$, the former is true.  The function $u\circ\pi\circ f$ attains its maximum on every connected component of $\nu\inv(U)$, so by the maximum principle for subharmonic functions on smooth, strictly analytic curves~\cite[Proposition~3.1.11]{thuillier05:thesis}, the function $u\circ\pi\circ f$ is constant on every connected component of $\nu\inv(U)$.  It follows that $u\circ\pi\circ\iota$ is constant on $U$.

  Now we reduce to the strictly analytic case handled above.  We may assume that $X$ is affinoid. As explained in \artref{art:shilov.section}, there is an analytic extension field $k'=k_r$ for some $r \in \R_{>0}^n$ such that $X_{k'}$ is strictly $k'$-analytic.  Let $\sigma\colon X\to X_{k'}$ be the Shilov section of $\pi_{k'/k}$.
  Then $x' = \sigma(x)\in\Int(X_{k'})$ because $\pi_{k'/k}\inv(\Int(X))\subset\Int(X_{k'})$ by~\artref{relative boundary}\eqref{boundary and base change}.  It is immediate from the definition that $u\circ\pi_{k'/k}$ is subharmonic, and it attains a local maximum at $x'$, so $u\circ\pi_{k'/k}$ is constant in an open neighborhood $U'$ of $x'$.  Then $u$ is constant on the open set $U = \sigma\inv(U')$.
\end{proof}

\begin{proof}[Proof of Proposition~\ref{prop:subharmonic.properties}(\ref{shp.scalars})]
  It is immediate from the definition that subharmonicity of $u$ implies subharmonicity of $u\circ\pi_{k'/k}$.  Suppose then that $u\circ\pi_{k'/k}$ is subharmonic.  Let $k''/k$ be another analytic extension field  with $k''$ non-trivially valued, and let $f\colon Y\to X_{k''}$ be a  morphism from a smooth, connected, strictly $k''$-analytic curve.  Assume that $u\circ\pi_{k''/k}\circ f\not\equiv-\infty$, so we need to show that $u\circ\pi_{k''/k}\circ f$ is subharmonic in Thuillier's sense.  Let $K$ be a non-Archimedean field that is simultaneously an analytic extension of $k'$ and $k''$~\artref{art:simultaneous.extensions}, and consider the following commutative diagram:
  \begin{equation}\label{eq:scalars.shp.diagram}
    \begin{tikzcd}
      Y_K \rar["f_K"] \dar & X_K \rar["\pi_{K/k'}"] \dar["\pi_{K/k''}"] & X_{k'} \dar["\pi_{k'/k}"] \\
      Y \rar["f"'] & X_{k''} \rar["\pi_{k''/k}"'] & X \rar["u"'] & \R\cup\{-\infty\}\rlap.
    \end{tikzcd}
  \end{equation}
  Since $u\circ\pi_{k'/k}$ is subharmonic, the composition $u\circ\pi_{K/k}\circ f_K$ is, on each connected component $C$ of $Y_K$, either subharmonic in Thuillier's sense or identically $-\infty$; since $C$ surjects onto $Y$ \cite[1.5.6]{ducros18:families} and $u\circ\pi_{k''/k}\circ f\not\equiv-\infty$, the latter does not happen.  Replacing $X$ by $Y$ and $k'/k$ by $K/k''$, we are reduced to showing that if $X$ is \emph{smooth} and strictly $k$-analytic for $k$ non-trivially valued and $u\circ\pi_{k'/k}$ is subharmonic in Thuillier's sense, then so is $u$.

  We can do this directly using Thuillier's definition.  (The converse is the difficult direction: see~\cite[Corollaire~3.4.5]{thuillier05:thesis}.)  Let $U\subset X$ be a strictly $k$-affinoid domain and let $h\colon U\to\R$ be a harmonic function such that $u\leq h$ on $\del U$.  Then $h\circ\pi_{k'/k}\colon U_{k'} = \pi_{k'/k}\inv(U)\to\R$ is harmonic by~\cite[Proposition~2.3.18]{thuillier05:thesis}.  We have $\pi_{k'/k}(\del U_{k'}) = \del U$ by Lemma~\ref{lem:base.change.boundaries}, so $u\circ\pi_{k'/k}\leq h\circ\pi_{k'/k}$ on $\del U_{k'}$.  Since $u\circ\pi_{k'/k}$ is subharmonic in Thuillier's sense, this implies $u\circ\pi_{k'/k}\leq h\circ\pi_{k'/k}$ on $U_{k'}$, so $u\leq h$ on $U$ since $\pi_{k'/k}\colon U_{k'}\surject U$ is surjective.  Thus $u$ is subharmonic in Thuillier's sense, as claimed.
\end{proof}

\begin{proof}[Proof of Proposition~\ref{prop:subharmonic.properties}(\ref{shp.functoriality})]
  It is immediate that subharmonicity of $u$ implies subharmonicity of $u\circ f$.  Suppose then that $f$ is finite and surjective and that $u\circ f$ is subharmonic.  Let $k'/k$ be an analytic extension field  with $k'$ non-trivially valued and let $g\colon Z\to X_{k'}$ be a morphism from a smooth, connected, strictly $k'$-analytic curve.  Suppose that $u\circ\pi_{k'/k}\circ g\not\equiv-\infty$, so we need to show that $u\circ\pi_{k'/k}\circ g$ is subharmonic in Thuillier's sense.  Let $W = Z\times_{X_{k'}} Y_{k'}$, and consider the following commutative diagram:
  \begin{equation}\label{eq:functoriality.shp.diagram}
    \begin{tikzcd}
      W \rar["g'"] \dar["f''"'] \drar[phantom, "\square"]
      & Y_{k'} \rar["\pi_{k'/k}"] \dar["f'"] & Y \dar["f"] \\
      Z \rar["g"'] & X_{k'} \rar["\pi_{k'/k}"'] & X \rar["u"'] & \R\cup\{-\infty\}\rlap.
    \end{tikzcd}
  \end{equation}
  The composition $(u\circ f)\circ\pi_{k'/k}\circ g'$ is subharmonic because it is pulled back from $u\circ f$ by an extension of scalars and the morphism $g'\colon W\to Y_{k'}$.  Since $f$ is finite and surjective, so are $f'$ and $f''$ by~\artref{art:simultaneous.extensions}, so we may replace $X$ by $Z$ and $Y$ by $W$ to assume $X$ is a \emph{smooth}, connected, strictly $k$-analytic curve  for $k$ non-trivially valued and that $u\not\equiv-\infty$.  We can check subharmonicity of $u$ (in Thuillier's sense) after extending scalars as in the proof of Proposition~\ref{prop:subharmonic.properties}(\ref{shp.scalars}), so we may assume in addition that $k$ is perfect.

  Let $\iota\colon Y_{\red}\inject Y$ be the inclusion of the underlying reduced space, and let $\nu\colon\td Y\to Y_{\red}$ be the normalization.  Since $k$ is perfect, the normal curve $\td Y$ is rig-smooth, and since $\td Y\to X$ is finite and $\del X=\emptyset$ (as $X$ is smooth), the curve $\td Y$ is smooth as well: see Lemmas~\ref{lem:reduce.to.reduced} and~\ref{lem:reduce.to.smooth}.  Both $\iota$ and $\nu$ are finite and surjective, so $f\circ\iota\circ\nu\colon\td Y\to X$ is again finite and surjective.  By Definition~\ref{def:subharmonicity} (with $k=k'$), subharmonicity of $u\circ f$ means that, on each connected component $C$ of $\td Y$, the function $(u\circ f)\circ\iota\circ\nu$ is either subharmonic in Thuillier's sense or is identically $-\infty$.  Since $C\to X$ is surjective (it is finite, hence it is a closed map, and it is an open map by~\cite[Lemma~3.2.4]{berkovic90:analytic_geometry}, so its image is all of $X$ since $X$ is connected), the latter does not happen, so $u\circ f\circ\iota\circ\nu$ is subharmonic in Thuillier's sense.    Thus we replace $Y$ by $\td Y$ and $f$ by $f\circ\iota\circ\nu$ to assume $X$ and $Y$ are both smooth, and that $u\circ f$ is subharmonic in Thuillier's sense.

  At this point, one proves that $u$ is subharmonic in Thuillier's sense directly from Thuillier's definition, as in the proof of Proposition~\ref{prop:subharmonic.properties}(\ref{shp.scalars}), this time using~\cite[Proposition~2.3.19]{thuillier05:thesis}.  The converse is the difficult direction: see~\cite[Proposition~3.1.14]{thuillier05:thesis}.
\end{proof}

We will need the following lemma in the proof of Proposition~\ref{prop:subharmonic.properties}(\ref{shp.-infty}).  (Note that Lemma~\ref{lem:-infty.on.boundary} itself follows from Proposition~\ref{prop:subharmonic.properties}(\ref{shp.-infty}) as we have seen in \artref{unique irreducible}.)

\begin{lem}\label{lem:-infty.on.boundary}
  Let $X$ be a strictly $k$-analytic curve over a non-trivially valued field $k$ and let $u\colon X\to\R\cup\{-\infty\}$ be a subharmonic function.  If $x\in\del X$ and $u(x)=-\infty$, then $u\equiv-\infty$ on a neighborhood of~$x$.
\end{lem}

\begin{proof}
  If $\chr(k)=p>0$ let $k'$ be the completion of $k^{p^{-\infty}}$, and otherwise let $k'=k$.  Let $X' = X_{k'}$ and $\pi = \pi_{k'/k}\colon X'\to X$, let $\iota\colon X'_{\red}\inject X'$ be the inclusion of the underlying reduced space, and let $\nu\colon\td X'\to X'_{\red}$ be the normalization.  Then $\pi$ and $\iota$ are homeomorphisms preserving boundaries by Lemmas~\ref{lem:reduce.to.perfect} and~\ref{lem:reduce.to.reduced}, and $\nu$ is a boundary-preserving isomorphism in a neighborhood of the unique point $x'\in\td X'$ lying over $x$ by Lemma~\ref{lem:reduce.to.smooth}.  The composition $u\circ\pi\circ\iota\circ\nu$ is subharmonic, so we may replace $X$ by $\td X'$ to assume that $X$ is rig-smooth.  After shrinking $X$, we may assume that $X$ is strictly affinoid.  By~\cite[Corollary~1.3.6]{berkovic90:analytic_geometry}, if $K$ is the completion of an algebraic closure of $k$ then $X_K\to X$ is an open map, and by Lemma~\ref{lem:base.change.boundaries} there is a point of $\del X_K$ lying over $x$, so we may extend scalars to assume that $k$ is algebraically closed.

  By the semistable reduction theorem~\cite[Th\'eor\`eme~2.3.8]{thuillier05:thesis}, there exists a semistable (formal) model $\fX$ of $X$.  Let $\red\colon X\to\fX_s$ denote the reduction map, which is anti-continuous.  Let $\xi\in\fX_s$ be any point.  As explained in~\cite[Theorem~4.6]{baker_payne_rabinoff13:analytic_curves} (see also \cite[Propositions 2.2 and 2.3]{bosch_lutkeboh85:stable_reduction_I}), if $\xi$ is a generic point then $\red\inv(\xi)$ is a single type-$2$ point of $X$; if $\xi$ is a smooth closed point then $\red\inv(\xi)$ is isomorphic to an open unit ball; and if $\xi$ is a node then $\red\inv(\xi)$ is isomorphic to an open annulus.  In particular, if $\xi$ is a closed point then $\red\inv(\xi)$ is connected and open.

  By~\cite[Lemme~6.5.1]{chambert_ducros12:forms_courants}, the closure of $\zeta = \red(x)$ is not proper, so $\zeta$ is a generic point of $\fX_s$.  Let $Y\subset\fX_s$ denote its closure.  Then $\red\inv(Y)$ is an open neighborhood of $x$ in $X$.  Let $\xi\in Y$ be a closed point and let $U = \red\inv(\xi)$.  Then $x$ is a limit point of $U$: one can use~\cite[Lemma~3.2]{baker_payne_rabinoff13:analytic_curves}, or one can see this using retraction to the skeleton $S(\fX)$ from~\cite[Th\'eor\`eme~2.2.10]{thuillier05:thesis}; in this case, $x$ is a vertex of $S(\fX)$, and either $U$ deformation retracts onto $x$ if $\xi$ is smooth, or $U$ contains an edge of $S(\fX)$ adjacent to $x$ when $\xi$ is a node.  It follows that $\red\inv(Y)$ is contained in the union of $\{x\}$ with all connected components $U$ of $\Int(X)$ having $x$ as a limit point.  We will show that $u\equiv-\infty$ on every such $U$.

   We adapt the proof of~\cite[Lemme~3.1.9]{thuillier05:thesis}.  Suppose that $u\not\equiv-\infty$ on some such $U$.  Then $u|_U$ is subharmonic in Thuillier's sense as $U$ is smooth and connected and $u$ is subharmonic in our sense.   By~\cite[Lemme~2.3.9]{thuillier05:thesis}, for every $N\in\Z_{\geq0}$, there exists a unique harmonic function $h_N\colon X\to\R$ such that $h_N(x) = -N$ and
  \[ h_N(y) =
    \begin{cases}
      u(y) & \text{if } u(y) > -\infty \\
      0 & \text{otherwise}
    \end{cases}
  \]
  for every $y \in \partial X \setminus \{x\}$.
  By construction, we have $u-h_N\leq 0$ on $\del X$.  By~\cite[Corollaire~3.1.12]{thuillier05:thesis}, if the function $u-h_N$ attains a local maximum at a point $y\in U$, then $u-h_N$ takes a constant value $C$ on $U$.  By semicontinuity, we have
  \[ -\infty = u(x)-h_N(x) \geq \limsup_{\substack{z\to x\\z\in U}}\bigl( u(z)-h_N(z) \bigr) = C, \]
  which is a contradiction.  Thus $u-h_N$ does not attain a local maximum on $U$.  It does attain a maximum on the closure $\bar U$, which is compact.  We have $\bar U\subset U\cup\del X$ and $u-h_N\leq 0$ on $\del X$, so $u-h_N\leq 0$ on $U$.  Clearly $-h_N(x)\to\infty$ as $N\to\infty$, so by~\cite[Lemme~2.3.10]{thuillier05:thesis}, we have $-h_N\to\infty$ on $U$.  Since $u\leq h_N$ on $U$ for all $N$, it follows that $u\equiv-\infty$ on $U$.
\end{proof}

\begin{proof}[Proof of Proposition~\ref{prop:subharmonic.properties}(\ref{shp.-infty})]
  We replace $X$ by $\Int(X)\cup\{x\}$ to assume that $\del X = \emptyset$ or $\del X = \{x\}$.  Replacing $X$ by the irreducible component of $X$ containing $x$ (see~\artref{unique irreducible}), we assume in addition that $X$ is irreducible.  Let $k'=k_r/k$ be an analytic extension field for some $r\in \R_{>0}^n$ as in~\artref{art:shilov.section} such that $X_{k'}$ is strictly $k'$-analytic.  Note that $X_{k'}$ is again irreducible: any irreducible component of $X_{k'}$ is defined over a finite, separable extension of $k$ inside $k'$~\cite[1.5.6]{ducros18:families}, but $k$ is algebraically closed in $k'$~\cite[3.2.11]{ducros14:structur_des_courbes_analytiq}.  The fiber $\pi_{k'/k}\inv(x)$ has a unique Shilov boundary point $x' = \sigma(x)$, so by Lemma~\ref{lem:fiber.boundary}, if $\del X = \{x\}$ then $\del X_{k'} = \{x'\}$; of course, if $\del X = \emptyset$ then $\del X_{k'} = \emptyset$ as well.  Replacing $X$ by $X_{k'}$ and $x$ by $x'$, we may assume that $X$ is strictly $k$-analytic and irreducible over a non-trivially valued field $k$, and that $\del X = \{x\}$ or $\del X = \emptyset$.

  If $\chr(k)=p>0$ let $k'$ be the completion of $k^{p^{-\infty}}$, and otherwise let $k'=k$.  Let $X' = X_{k'}$ and let $\pi = \pi_{k'/k}\colon X'\to X$, let $\iota\colon X'_{\red}\inject X'$ be the inclusion of the underlying reduced space, and let $\nu\colon\td X'\to X'_{\red}$ be the normalization.  Since $\pi$ and $\nu$ are homeomorphisms, the space $X'_{\red}$ is again irreducible, so $\td X'$ is connected by~\cite[Proposition~5.16]{ducros09:excellent}.  Since $x$ is not a rig-point of $X$, there exists a unique point $x'\in\td X'$ lying over $x$ by Lemmas~\ref{lem:reduce.to.perfect}, \ref{lem:reduce.to.reduced} and~\ref{lem:reduce.to.smooth}.  We must show that $u\circ\pi\circ\iota\circ\nu\equiv-\infty$ on $\td X'$.  Since $u\circ\pi\circ\iota\circ\nu$ is subharmonic by Proposition~\ref{prop:subharmonic.properties}(\ref{shp.scalars},\ref{shp.functoriality}), we may replace $X$ by $\td X'$ to assume that $X$ is rig-smooth and connected (and still that $\del X = \emptyset$ or $\del X = \{x\}$ using the quoted lemmas again).

  If $\del X=\emptyset$ then $X$ is smooth and connected, so $u$ is either identically $-\infty$ or it is subharmonic in Thuillier's sense.  The latter case is impossible when $u(x)=-\infty$ by~\cite[Proposition~3.4.10]{thuillier05:thesis}.

  Suppose then that $X$ is rig-smooth and connected, and that $\del X=\{x\}$ and $u(x)=-\infty$.  By Lemma~\ref{lem:-infty.on.boundary}, there is an open neighborhood $U$ of $x$ on which $u\equiv-\infty$.  Since $X$ is connected, every connected component of $\Int(X)$ contains $x$ in its closure.  Thus $U$ intersects every connected component $C$ of $\Int(X)$.  This implies that $u$ takes the value $-\infty$ on a non-rig point of $C$, so the previous paragraph shows that $u\equiv-\infty$ on $C$.  It follows that $u\equiv-\infty$ on $X$.
\end{proof}

\begin{art} \label{art:harmonic.is.pm.subharmonic}
  Now we turn to harmonic functions.  Let $X$ be a smooth, strictly $k$-analytic curve  with $k$ non-trivially valued and let $h\colon X\to\R$ be a function.  We claim that $h$ is harmonic in the sense of~(\ref{sec:thuillier.harmonic}) if and only if $h$ and $-h$ are subharmonic (in our sense or Thuillier's---see~(\ref{art:subharmonicity.real.values})).  This is implicit in Thuillier's thesis, but it requires justification.  Suppose then that $h\colon X\to\R$ is harmonic.  Then $\pm h$ are subharmonic by~\cite[Proposition~2.3.13 and Corollaire~3.1.12]{thuillier05:thesis}.  Conversely, suppose that $\pm h$ are subharmonic.  This implies that $h$ is continuous.  Let $U\subset X$ be a strictly affinoid domain.  By~\cite[Lemme~2.3.9]{thuillier05:thesis}, there exists a unique harmonic function $h'\colon U\to\R$ that coincides with $h$ on $\del U$.  By definition of subharmonicity, we have $h\leq h'$ on $U$.  By the same argument as applied to $-h$, we see that $h=h'$ on $U$, so that $h$ is harmonic on $U$.  As $U$ was arbitrary, $h$ is harmonic.
\end{art}

By~\artref{art:harmonic.is.pm.subharmonic}, the following definition recovers Thuillier's definition of harmonicity on a smooth, strictly $k$-analytic curve, and even on a strictly affinoid rig-smooth curve using~\cite[Corollaire~2.3.14]{thuillier05:thesis}.

\begin{defn}[Harmonic functions]\label{def:harmonicity}
  Let $X$ be a $k$-analytic curve for a possibly trivially valued non-Archimedean field $k$.  A function $h\colon X\to\R$ is \defi{harmonic} if $h$ and $-h$ are both subharmonic.
\end{defn}

The advantage of this definition is that the following basic properties of harmonic functions follow immediately from the corresponding properties of subharmonic functions.

\begin{prop}\label{prop:harmonic.properties}
  Let $X$ be a $k$-analytic curve and let $h\colon X\to\R$ be a continuous function.
  \begin{enumerate}
  \item \label{hp.sheafiness} \textup{(Sheafiness)} The harmonic functions form a sheaf on $X$.
  \item \label{hp.analytic}~ If $f\in\Gamma(X,\sO_X^\times)$ then $\log|f|$ is harmonic.
  \item \label{hp.vectorspace} The harmonic functions on $X$ form a real vector space.
  \item \label{hp.limits} \textup{(Limits)} If $h$ is locally a uniform
   limit of harmonic functions, then it is harmonic.
  \item \label{hp.maximum} \textup{(Maximum Principle)} If $h$ is harmonic and $h$ attains a local extremum at $x\in\Int(X)$, then $h$ is constant in a neighborhood of $x$.
  \item \label{hp.scalars} \textup{(Extension of Scalars)} Let $k'/k$ be an analytic extension field.  Then $h$ is harmonic if and only if $h\circ\pi_{k'/k}\colon X_{k'}\to\R$ is harmonic.
  \item \label{hp.functoriality} \textup{(Functoriality)} Let $Y$ be a $k$-analytic curve and let $f\colon Y\to X$ be a morphism.  If $h$ is harmonic then $h\circ f$ is harmonic, and the converse holds if $f$ is finite and surjective.
  \item \label{hp.boundaries} \textup{(Boundaries)} If $h$ is harmonic on $\Int(X)$, then it is harmonic on $X$.
  \end{enumerate}
\end{prop}

We will later see in Theorem \ref{pointwise convergence of pluriharmonic} that \eqref{hp.limits} holds even for pointwise limits.
The following proposition asserts that Definition~\ref{def:harmonicity} is equivalent to the notion of harmonicity in~\cite[Definition~7.2]{gubler_rabinoff_jell:harmonic_trop}. We assume that $k$ is non-trivially valued as in~\textit{ibid}.

\begin{prop}\label{prop:grj.harmonicity.curves}
  Let $X$ be a strictly $k$-analytic curve  and let $h\colon X\to\R$ be a continuous function.  The function $h$ is harmonic if and only if it is $\R$-PL and the line bundle $L_h(x)\in\Pic(\red(X,x))_\R$ defined in~(\ref{art:reduction.of.germs}) is numerically trivial for every point $x\in X$.
\end{prop}

\begin{proof}
This is~\cite[Proposition~15.7]{gubler_rabinoff_jell:harmonic_trop} in the rig-smooth case; we reduce to this case in the usual way.  If $\chr(k)=p>0$ let $k'$ be the completion of $k^{p^{-\infty}}$, and otherwise let $k'=k$.  Let $X' = X\hat\tensor_k k'$ and let $\pi\colon X'\to X$ be the structure morphism.  Let $\iota\colon X'_{\red}\inject X'$ be the underlying reduced space and let $\nu\colon\td X'\to X'_{\red}$ be the normalization, so that $\td X'$ is rig-smooth.  Let $f\colon\td X'\to X$ denote the composition $\pi\circ\iota\circ\nu$.

  Suppose that $h$ is harmonic.  Then $h\circ f$ is harmonic by Proposition~\ref{prop:harmonic.properties}(\ref{hp.scalars},\ref{hp.functoriality}), so $h\circ f$ is $\R$-PL and $L_h(x')$ is numerically trivial for all $x'\in\td X'$ by the rig-smooth case.  Then $h$ is $\R$-PL by Lemma~\ref{lem:reduce.to.smooth}(6) and~\cite[Lemma~5.3(3), Corollary~5.11]{gubler_rabinoff_jell:harmonic_trop}, and $L_h(x)$ is numerically trivial for all $x\in X$ by~\cite[Proposition~7.6($1',2'$)]{gubler_rabinoff_jell:harmonic_trop}.  Conversely, suppose that $h$ is $\R$-PL and that $L_h(x)$ is numerically trivial for all $x\in X$.  The same holds for $h\circ f$ by~\cite[Proposition~7.6($1',2'$)]{gubler_rabinoff_jell:harmonic_trop}, so $h\circ f$ is harmonic by the rig-smooth case, and hence $h$ is harmonic by Proposition~\ref{prop:harmonic.properties}(\ref{hp.scalars},\ref{hp.functoriality}).
\end{proof}

\begin{rem}
  We still assume that $X$ is a strictly $k$-analytic curve over a non-trivially valued field $k$ and that $h\colon X \to \R$ is continuous. Proposition~\ref{prop:grj.harmonicity.curves} and Proposition~\ref{prop:harmonic.properties}(\ref{hp.scalars}) imply in particular that if $h\circ\pi_{k'/k}$ is harmonic for some analytic extension field  $k'/k$, then $h\circ\pi_{k''/k}$ is $\R$-PL for \emph{any} analytic extension $k''/k$, which is by no means obvious.  Likewise, Proposition~\ref{prop:harmonic.properties}(\ref{hp.functoriality}) implies that if $f\colon Y\to X$ is finite and surjective and $h\circ f$ is harmonic, then $h$ is $\R$-PL, and Proposition~\ref{prop:harmonic.properties}(\ref{hp.boundaries}) implies that if $h|_{\Int(X)}$ is harmonic then $h$ is $\R$-PL.
\end{rem}

\section{Classically Plurisubharmonic functions} \label{section:psh functions}

In this section we define the notion of a \defi{classically plurisubharmonic function}.
The definition follows Definition~\ref{def:subharmonicity}.

\begin{defn}[Classically Plurisubharmonic Functions]\label{def:classically.psh}
  Let $X$ be a $k$-analytic space and let $u\colon X\to\R\cup\{-\infty\}$ be an upper semicontinuous function.  We say that $u$ is \defi{classically plurisubharmonic} or \defi{classically psh} provided that, for every analytic extension field $k'/k$, every $k'$-analytic curve $Y$, and every morphism $f\colon Y\to X_{k'}$, the function $u\circ\pi_{k'/k}\circ f\colon Y\to\R\cup\{-\infty\}$ is subharmonic on $Y$ in the sense of Definition~\ref{def:subharmonicity}.
\end{defn}

\begin{artsub}
  Any class of psh functions should respect extension of scalars and should be stable under pullback, so the conditions in Definition~\ref{def:classically.psh} are necessary.  In complex analysis, this condition (without the extension of scalars) is the definition of plurisubharmonicity, which is why we call the notion ``classically psh''.
\end{artsub}

\begin{artsub}\label{art:psh.is.subharmonic}
  If $X$ is a $k$-analytic curve, then a function $u\colon X\to\R\cup\{-\infty\}$ is classically psh if and only if it is subharmonic.  Indeed, if $u$ is classically psh then taking $k'=k$ and $Y=X$ in Definition~\ref{def:classically.psh} shows that $u$ is subharmonic, and if $u$ is subharmonic then it is classically psh by Proposition~\ref{prop:subharmonic.properties}(\ref{shp.scalars},\ref{shp.functoriality}).
\end{artsub}

\begin{artsub}\label{art:check.on.smooth}
  In Definition~\ref{def:classically.psh}, one can assume that $Y$ is a connected, smooth, strictly $k'$-analytic curve: this follows from  Definition~\ref{def:subharmonicity}.
\end{artsub}

Classically psh functions satisfy the basic properties of subharmonic functions on curves.

\begin{prop}\label{prop:classically.psh.properties}
  Let $X$ be a $k$-analytic space.
  \begin{enumerate}
  \item\label{psh.sheafiness} \textup{(Sheafiness)} The classically psh functions form a sheaf on $X$.
  \item\label{psh.analytic} If $f\in\Gamma(X,\sO_X)$ then $\log|f|\colon X\to\R\cup\{-\infty\}$ is classically psh.
  \item\label{psh.cone} \textup{(Cone Property)} If $u_1,u_2\colon X\to\R\cup\{-\infty\}$ are classically psh and $\lambda_1,\lambda_2\in\R_{\geq0}$ then $\lambda_1u_1+\lambda_2u_2$ and $\max\{u_1,u_2\}$ are classically psh.
  \item\label{psh.limits} \textup{(Limits)} The infimum of a decreasing net of classically psh functions is classically psh. For arbitrary pointwise limits, we refer to Theorem \ref{classically psh and pointwise limits}.
  \item\label{psh.scalars} \textup{(Extension of Scalars)} Let $k'/k$ be an analytic extension field.  A function $u\colon X\to\R\cup\{-\infty\}$ is classically psh if and only if $u\circ\pi_{k'/k}\colon X_{k'}\to\R\cup\{-\infty\}$ is classically psh.
  \item\label{psh.functoriality} \textup{(Functoriality)} Let $Y$ be a $k$-analytic space and let $f\colon Y\to X$ be a morphism.  If $u\colon X\to\R\cup\{-\infty\}$ is classically psh then $u\circ f$ is classically psh, and the converse holds if $f$ is finite and surjective.
  \end{enumerate}
\end{prop}

\begin{proof} Properties (\ref{psh.sheafiness},\ref{psh.cone},\ref{psh.limits}) follow immediately from Proposition~\ref{prop:subharmonic.properties}(\ref{shp.sheafiness},\ref{shp.cone},\ref{shp.limits}). For Proposition~\ref{prop:classically.psh.properties}(\ref{psh.analytic}), we consider an analytic function $f$ on $X$. As $|f|$ is continuous, we deduce that $\log|f|$ is a continuous function with values in $\R \cup \{-\infty\}$ and hence it is upper semicontinuous. Then Proposition~\ref{prop:classically.psh.properties}(\ref{psh.analytic}) follows from Proposition~\ref{prop:subharmonic.properties}(\ref{shp.analytic}).

	 To prove 	Proposition~\ref{prop:classically.psh.properties}(\ref{psh.scalars}), it is clear from the definitions that if $u$ is classically psh, then $u \circ \pi_{k'/k}$ is classically psh for every analytic extension field $k'/k$. To prove the converse, we assume that $k'/k$ is an analytic  extension field and that $u'\coloneqq u \circ \pi_{k'/k}$ is a classically psh function on $X_{k'}$. Since the  morphism $\pi_{k'/k}$ is closed and surjective \artref{extension of scalars}, it is clear that $u$ is upper semicontinuous.
  Let $k''/k$ be an analytic  extension field and let $f\colon Y \to X_{k''}$ be a morphism over $k''$ from a $k''$-analytic curve $Y$. We have  to show that $u \circ \pi_{k''/k} \circ f$ is subharmonic.  As in the proof of Proposition~\ref{prop:subharmonic.properties}(\ref{shp.scalars}), we consider a simultaneous analytic  extension field $K$ of $k'$ and $k''$. It follows from the above that $u \circ \pi_{K/k}= u' \circ \pi_{K/k'}$ is classically psh and hence $u \circ \pi_{K/k} \circ f_K$ is subharmonic.   We conclude from the commutative diagram~\eqref{eq:scalars.shp.diagram} that  $u \circ \pi_{K/k} \circ f_K= u \circ \pi_{k''/k} \circ f \circ \pi_{K/k''}$ is subharmonic and hence $u \circ \pi_{k''/k} \circ f$ is subharmonic by Proposition~\ref{prop:subharmonic.properties}(\ref{shp.scalars}).

	Likewise, the proof of Proposition~\ref{prop:classically.psh.properties}(\ref{psh.functoriality}) follows the same lines as the proof of Proposition~\ref{prop:subharmonic.properties}(\ref{shp.functoriality}).
	Again, it follows from the definitions if $u$ is a classically psh function on $X$, then $u \circ f$ is a classically psh function on $Y$. Conversely, we assume that $u \circ f$ is a classically psh function for a finite surjective morphism $f\colon Y \to X$. Since finite morphisms are proper and hence closed, we deduce from upper semicontinuity of $u \circ f$ and from surjectivity of $f$ that $u$ is upper semicontinuous. It remains to show that $u \circ \pi_{k'/k} \circ g$ is subharmonic for any analytic  extension field $k'/k$ and any morphism $g\colon Z \to X_{k'}$ from a connected $k'$-analytic curve $Z$.
	Setting $W \coloneqq Z \times_{X_{k'}} Y_{k'}$, we get a morphism $g' \colon W \to Y_{k'}$ from $f$ by base change with respect to $f'\colon Y_{k'} \to X_{k'}$ as in~\eqref{eq:functoriality.shp.diagram}. Since $u  \circ f$ is classically psh, it follows from Proposition~\ref{prop:classically.psh.properties}(\ref{psh.scalars}) that $u \circ \pi_{k'/k} \circ f'= u \circ f \circ \pi_{k'/k}$ is classically psh. The morphism $f$ is finite and surjective, hence the same holds for $f'$. We conclude that the base change $f''\colon W \to Z$ of the morphism $f'$ is finite and surjective. The $k'$-analytic space $W$ has irreducible components of dimension $1$ and maybe isolated irreducible components of dimension $0$. Removing the latter, we obtain a $k'$-analytic curve $W'$ as an open and closed subset of $W$. Since $Z$ is a connected $k'$-analytic curve and $f''$ is a finite surjective map from $W$ onto $Z$, we get an induced finite surjective morphism $W' \to Z$. Indeed, as $W'$ is also Zariski-closed in $W$, we deduce from finiteness of $f''$ that $f''(W')$ is a Zariski-closed subset in $Z$ of dimension $1$ and hence $f''(W')=Z$ as $Z$ is a connected $k'$-analytic curve. Since the inclusion $W' \to W$ is obviously finite, the composition $W' \to Z$ is also finite. In the following, we replace $W$ by $W'$, which preserves the commutativity of the diagram~\eqref{eq:functoriality.shp.diagram} and has the advantage that $W$ is a $k'$-analytic curve.
        It follows that $u \circ \pi_{k'/k} \circ f' \circ g'=u \circ \pi_{k'/k}\circ g\circ f''$ is subharmonic on $W$.
	The morphism of $k'$-analytic curves $f''\colon W \to Z$ is finite and surjective, so we deduce from Proposition~\ref{prop:subharmonic.properties}(\ref{shp.functoriality}) that $u \circ \pi_{k'/k}\circ g$ is subharmonic.
	\end{proof}

        In the paper \cite{GR25semipositivePL}, we have studied semipositive $\R$-PL functions on a good strictly $k$-analytic space over a non-trivially valued non-archimedean field $k$.  This turns out to be the same as an $\R$-PL function that is classically psh, as we now show.

        \begin{thm} \label{PL-functions and classically psh}
          Suppose that the valuation on $k$ is nontrivial.  Let $X$ be a good, strictly $k$-analytic space and let $h\colon X \to \R$ be an $\R$-PL function. Then $h$ is semipositive if and only if $h$ is classically psh.
        \end{thm}

\begin{proof}
  For $x \in X$, an $\R$-PL function $h \colon X \to \R$ induces a residue line bundle $L_h(x) \in \Pic(\red(X,x))_\R$ on the reduction $\red(X,x)$ of the germ $(X,x$): see \artref{art:reduction.of.germs}.  Recall from \cite[Definition 7.2]{gubler_rabinoff_jell:harmonic_trop} that $h$ is \emph{semipositive} at $x \in X$ if and only if the associated residue line bundle $L_h(x)$ is nef. We have shown in \cite[Proposition 3.1]{GR25semipositivePL} that for a smooth strictly $k$-analytic curve, an $\R$-PL function is semipositive if and only if it is subharmonic.  	Since semipositivity is stable under base extension and pull-back \cite[Proposition 7.6]{gubler_rabinoff_jell:harmonic_trop}, we conclude that a semipositive $\R$-PL function on $X$ is classically psh.

  The converse follows from \cite[Corollary 3.5]{GR25semipositivePL}, which shows that an $\R$-PL function $h$ on $X$ is semipositive if and only if there is an algebraically closed analytic extension field $F$ of $k$ such that for every smooth $F$-analytic curve $C$  and every morphism $\varphi\colon C \to X_F$, we have $h \circ \varphi$ semipositive.
\end{proof}

\section{Connectivity by curves} \label{section:connectivity.by.curves}

We will prove a maximum principle for classically psh functions by reducing to the case of curves.  For this we will need to prove that it is possible to connect a point of an analytic space to its neighbors using analytic curves, potentially defined over large field extensions.

\begin{defn}\label{def:analytic.curve}
  Let $X$ be a $k$-analytic space.  An \defi{analytic curve in $X$} is a subset $\Gamma\subset X$ satisfying the following property: there is an analytic  extension field $k'/k$, a compact, separated  $k'$-analytic space $C$ of dimension at most $1$, and a morphism $\phi\colon C\to X_{k'}$, such that $\Gamma = \pi_{k'/k}\circ\phi(C)$.

  An analytic curve $\Gamma$ in $X$ is \defi{overconvergent} if there exist $(k',C,\phi)$ as above such that $\phi$ factors as
  \begin{equation}\label{eq:overconvergent}
    \begin{tikzcd}
      C \rar["\psi"] \drar["\phi"'] & \Int(C') \rar[hook] & C' \dlar["\phi'"] \\
      & X_{k'}
    \end{tikzcd}
  \end{equation}
  where $C'$ is a separated $k'$-analytic space of dimension at most $1$, and $\phi'$ and $\psi$ are morphisms.
\end{defn}

\begin{artsub}
  The compactness hypothesis in Definition~\ref{def:analytic.curve} is crucial.  When $X$ is Hausdorff, it implies that if $\Gamma,\Gamma'$ are analytic curves in $X$ such that $\Gamma\cup\Gamma'$ is connected, then $\Gamma\cap\Gamma'\neq\emptyset$ (the compact sets $\Gamma,\Gamma'$ are closed, so they cannot be disjoint by connectedness).  However, a non-proper compact curve has a nonempty boundary, which causes problems when trying to reduce the maximum principle to the case of curves.  This is the reason for introducing overconvergent curves in $X$.
\end{artsub}

\begin{artsub}\label{art:larger.extension}
  One can always pass to a larger analytic extension field $k''/k'$ in Definition~\ref{def:analytic.curve}. Indeed, we have a commutative diagram
  \[
    \begin{tikzcd}
      C_{k''} \rar["\phi_{k''}"] \dar[two heads] & X_{k''} \drar["\pi_{k''/k}"] \dar["\pi_{k''/k'}"'] \\
      C \rar["\phi"'] & X_{k'} \rar["\pi_{k'/k}"'] & X\rlap.
    \end{tikzcd}
  \]
  In the overconvergent case, extending scalars in~\eqref{eq:overconvergent} gives a commutative diagram
  \[
    \begin{tikzcd}
      C_{k''} \drar["\phi_{k''}"'] \rar["\psi_{k''}"] & \Int(C')_{k''} \rar[hook]
      & \Int(C'_{k''}) \rar[hook] & C'_{k''} \arrow[dll, bend left=10, "\psi_{k''}"] \\
      & X_{k''}
    \end{tikzcd}
  \]
  where $\pi_{k''/k'}\inv(\Int(C')) = \Int(C')_{k''}\inject\Int(C'_{k''})$ by~\artref{relative boundary}\eqref{boundary and base change}.  In particular, we may always assume that $C$ (and $C'$) are \emph{strictly} $k'$-analytic.
\end{artsub}

We say that an analytic curve $\Gamma$ in $X$ is \defi{connected} if it is connected as a topological space.

\begin{lem}\label{lem:analytic.curve.properties}
  Let $X$ be a $k$-analytic space and let $\Gamma\subset X$ be an analytic curve in $X$.
  \begin{enumerate}
  \item\label{acp.components}
    The connected components of $\Gamma$ are open in $\Gamma$ and are also analytic curves in $X$.
  \item\label{acp.union} If $\Gamma'\subset X$ is another analytic curve in $X$, then $\Gamma\cup\Gamma'$ is an analytic curve in~$X$.
  \item\label{acp.image} If $f\colon X\to Y$ is a morphism of $k$-analytic spaces, then $f(\Gamma)$ is an analytic curve in~$Y$.
  \item\label{acp.pullback} If $f\colon Y\to X$ is a finite morphism of $k$-analytic spaces, then $f\inv(\Gamma)$ is an analytic curve in $Y$.
  \item\label{acp.scalars} Let $k'/k$ be an analytic extension field. If $\Gamma'$ is an analytic curve in~$X_{k'}$, then $\pi_{k'/k}(\Gamma')$ is an analytic curve in~$X$.
  \end{enumerate}
  These assertions are also valid after replacing ``analytic curve'' with ``overconvergent analytic curve''.
\end{lem}

\begin{proof}
  In each part except~(\ref{acp.scalars}) we fix $(k',C,\phi)$ as in Definition~\ref{def:analytic.curve} such that $\Gamma = \pi_{k'/k}\circ\phi(C)$.  If $\Gamma$ is overconvergent then we also fix $(C',\phi',\psi)$ as in the second part of the definition.

  \medskip(\ref{acp.components})\;
  Let $C_1,\ldots,C_n$ be the connected components of $C$, and let $\Gamma_i = \pi_{k'/k}\circ\phi(C_i)$.  Note that each $\Gamma_i$ is compact and connected.  Hence any connected component $\Sigma$ of $\Gamma$ is the union of all $\Gamma_i$ that meet $\Sigma$.  We have $\Sigma = \pi_{k'/k}\circ\phi\bigl(\Djunion_{\Sigma\cap\Gamma_i\neq\emptyset} C_i \bigr)$, which is a connected analytic curve in $X$.  The complement of $\Sigma$ in $\Gamma$ is the union of all $\Gamma_i$ disjoint from $\Sigma$, so that $\Sigma$ is open and closed in $\Gamma$. We conclude that the analytic curve $\Sigma$ in $X$ is a connected component of $\Gamma$.

  \medskip(\ref{acp.union})\;
  Choose $(k'',D,\eta)$ as in Definition~\ref{def:analytic.curve} such that $\Gamma' = \pi_{k''/k}\circ\eta(D)$.  Passing to a common analytic extension field of $k'$ and $k''$~\artref{art:simultaneous.extensions}, we may assume $k'=k''$ by~\artref{art:larger.extension}.  Then $\Gamma\cup\Gamma'$ is the image of $C\djunion D$ under $\phi\djunion\eta$.  The overconvergent case follows similarly by taking disjoint unions.

  \medskip(\ref{acp.image})\;
  This is immediate from the definitions.

  \medskip(\ref{acp.pullback})\;
  Consider the commutative diagram
  \[
    \begin{tikzcd}
      {D} \rar["\eta"] \dar["g"'] \drar[phantom, "\square"] & Y_{k'} \rar["\pi_{k'/k}"] \dar["f_{k'}"] & Y \dar["f"] \\
      C \rar["\phi"'] & X_{k'} \rar["\pi_{k'/k}"'] & X
    \end{tikzcd}
  \]
  in which $D = C\times_{X_{k'}}Y_{k'}$.  Then $g\colon D\to C$ is finite, so $D$ is a compact, separated $k'$-analytic space of dimension at most $1$.  Let $\Gamma' = \pi_{k'/k}\circ\eta(D)$.  This is an analytic curve in $Y$ contained in $f\inv(\Gamma)$; we claim $\Gamma' = f\inv(\Gamma)$.  Let $x\in\Gamma$ and let $y\in f\inv(x)$.  Choose $z\in C$ such that $\pi_{k'/k}\circ\phi(z) = x$, and let $x' = \phi(z)$.  There exists $y'\in Y_{k'}$ such that $\pi_{k'/k}(y') = y$ and $f_{k'}(y') = x'$ because $\sH(y)\hat\tensor_{\sH(x)}\sH(x')\neq 0$~\artref{art:simultaneous.extensions}.  Likewise there exists $w\in D$ such that $\eta(w)=y'$ and $g(w)=z$; for such $w$ we have $\pi_{k'/k}\circ\eta(w) = y$, which proves $\Gamma' = f\inv(\Gamma)$.

  Suppose now that $\Gamma$ is overconvergent.  Consider the commutative diagram
  \[
    \begin{tikzcd}
      D \rar["\psi'"] \dar["g"']
      \arrow[rrr, "\eta", rounded corners, to path={
        -- ([yshift=2ex]\tikztostart.north)
        -- ([yshift=2ex]\tikztotarget.north)\tikztonodes
        -- (\tikztotarget)
      }]
      & \Int(D') \dar \rar[hook] \drar[phantom, "\square"]
      & D' \dar \rar["\eta'"] \drar[phantom, "\square"]
      & Y_{k'} \dar["f_{k'}"] \\
      C \rar["\psi"']
      \arrow[rrr, "\phi"', rounded corners, to path={
        -- ([yshift=-2ex]\tikztostart.south)
        -- ([yshift=-2ex]\tikztotarget.south)\tikztonodes
        -- (\tikztotarget)
      }]
      & \Int(C') \rar[hook] & C' \rar["\phi'"'] & X_{k'}
    \end{tikzcd}
  \]
  in which $D' = C'\times_{X_{k'}}Y_{k'}$.  Again $D'$ is a separated $k'$-analytic space of dimension at most~$1$.  The middle square is Cartesian by properties \eqref{finite and boundary} and \eqref{transitivity of boundaries} of~\artref{art:relative.interior}, which implies that there is a unique morphism $\psi'\colon D\to\Int(D')$ making the left square commute.  Hence $\Gamma'$ is overconvergent as well.

  \medskip(\ref{acp.scalars})\;
  Choose $(k'',C,\phi)$ as in Definition~\ref{def:analytic.curve} such that $\Gamma' = \pi_{k''/k'}\circ\phi(C)$.  Then
  \[ \pi_{k'/k}(\Gamma') = \pi_{k'/k}\circ\pi_{k''/k'}\circ\phi(C) = \pi_{k''/k}\circ\phi(C) \]
  under the identification $X_{k''} = (X_{k'})_{k''}$.  If $\Gamma'$ is overconvergent then clearly $\pi_{k'/k}(\Gamma')$ is overconvergent.
\end{proof}

\begin{defn}\label{def:conn.by.curves}
  We say that $X$ is \defi{connected by curves} (resp.\ \defi{connected by overconvergent curves}) if, for every pair of points $x,y\in X$, there exists a connected (resp.\ connected and overconvergent) analytic curve $\Gamma$ in $X$ such that $x,y\in\Gamma$.
\end{defn}

Berkovich~\cite[Theorem~3.2.1]{berkovic90:analytic_geometry} proved that a connected $k$-analytic space is path-connected.  The paths that he constructs (or rather, a variant of these paths) are in fact contained in analytic curves, as we will explain below.  Therefore, the following fact should be regarded as an extension of Berkovich's theorem.

\begin{thm}\label{thm:conn.by.curves}
  Let $X$ be a connected $k$-analytic space.  Then $X$ is connected by curves, and if $\del X = \emptyset$ then $X$ is connected by overconvergent curves.
\end{thm}

This result can be found in different forms in the literature.  De Jong proves in~\cite[Proposition~6.1.1]{dejong95:formalrigid} that any two rig-points of a strictly $k$-analytic space over a discretely valued field can be connected by strictly $k$-affinoid curves.  Berkovich proves in~\cite[Theorem~4.1.1]{berkovic07:integration} that any two rig-points in a connected, boundaryless $k$-analytic space (for arbitrary $k$) can be connected by boundaryless $k$-analytic curves (of a special type).  Using some additional arguments involving base change to a larger analytic field, the second part of our Theorem~\ref{thm:conn.by.curves} can be derived from this result.  Here we give a relatively elementary proof that works in both cases.

In order to prove Theorem~\ref{thm:conn.by.curves}, we construct some paths and analytic curves in $k$-analytic spaces, following Berkovich.  See~\artref{polydiscs and balls} for the notations $\bB^n, \bB^n_+$, etc.

\begin{art}[Paths in $\bB^1$]\label{art:paths.in.B1}
  For $x,y\in\bB^1=\sM(k\angles{T})$ we let $[x,y]$ denote the unique path from $x$ to $y$.  This can be described as follows.  Let $m(x,y)$ be the maximal point of the smallest closed ball containing $x$ and $y$, and let $r(x,y)$ be its radius.  Then $[x,m(x,y)]$ is the union of $\{x\}$ and the maximal points of all closed balls $\bB^1(x,\rho)$ containing $x$ with radius $\rho\leq r(x,y)$, and likewise for $[m(x,y),y]$.  The path $[x,y]$ is the concatenation $[x,m(x,y)]\cup[m(x,y),y]$.
\end{art}

\begin{art}[Paths in relative balls]\label{art:paths.in.Y}
  Let $X = \sM(\sA)$ be a $k$-affinoid space, let $Y = X\times\bB^1 = \sM(\sA\angles T)$, and let $\phi\colon Y\to X$ be projection onto the first factor.  For $x\in X$ and $y_1,y_2\in\phi\inv(x)$ we let $[y_1,y_2]$ denote the unique path from $y_1$ to $y_2$ in the fiber $\phi\inv(x)\cong\bB^1_{\sH(x)}$.  The inclusion $\bB^1_{\sH(x)}\cong\phi\inv(x)\inject Y$ is a connected analytic curve in $Y$: there is a canonical isomorphism $\phi\inv(x)\cong(Y\hat\tensor_k\sH(x))\times_{X\hat\tensor_k\sH(x)}\sM(\sH(x))$, with  $\phi\inv(x)\inject Y$ corresponding to projection onto the first factor followed by the structure morphism $Y\hat\tensor_k\sH(x)\to Y$.
\end{art}

\begin{art}[Paths to the Gauss point]
  Berkovich recursively uses the construction in~\ref{art:paths.in.Y} to connect any point in $\bB^n$ to the Gauss point $\eta_n\in\bB^n$.  Explicitly, in the case $n=2$, we let $\phi\colon\bB^2\to\bB^1$ denote projection onto the first factor, as above.  Let $\sigma\colon\bB^1\to\bB^2$ be the section of $\phi$ sending $x\in\bB^1$ to the Gauss point of $\phi\inv(x)$, and note that $\sigma(\eta_1) = \eta_2$.  For $x\in\bB^2$, the path from $x$ to $\eta$ is defined to be $[x,\sigma(\phi(x))]\cup\sigma([\phi(x),\eta_1])$.

  These paths are not suitable for proving that $\bB^n$ is connected by curves: the section $\sigma$ is continuous, but it is not a morphism of analytic spaces, so while $[\phi(x),\eta_1]$ is contained in the analytic curve $\bB^1$, its image under $\sigma$ is not an analytic curve in $\bB^2$.
\end{art}

\begin{art}[Paths to zero]\label{art:paths.in.Bn}
  Instead, we recursively define a path from a point $x\in\bB^n$ to the origin $0$.  When $n=1$ this is just the path $[x,0]$ defined in~\ref{art:paths.in.B1}.  In general, let $\phi\colon\bB^n=\bB^{n-1}\times\bB^1\to\bB^{n-1}$ be the projection onto the first factor, and let $\sigma\colon\bB^{n-1}\to\bB^n$ be the inclusion $\bB^{n-1}\cong\bB^{n-1}\times\{0\}\inject\bB^n$ (this is the fiber over $0$ of the projection onto the second factor $\bB^{n-1}\times\bB^1\to\bB^1$).  Then $\sigma$ is a section of $\phi$.  We define $[x,0]$ to be the concatenation of the path $[x,\sigma(\phi(x))]$ from~\ref{art:paths.in.Y} and $\sigma([\phi(x),0])$, where $[\phi(x),0]\subset\bB^{n-1}$ is defined from the recursion.
\end{art}

\begin{lem}\label{lem:paths.in.curves}
  For any point $x\in\bB^n$, the path $[x,0]$ is contained in a connected analytic curve in $\bB^n$.  If $x\in\bB^n_+$ then $[x,0]$ is contained in a connected, overconvergent analytic curve in $\bB^n_+$.
\end{lem}

\begin{proof}
  We prove both assertions by induction on $n$.  When $n=1$, the path $[x,0]$ is contained in the connected analytic curve $\bB^1$.  If $x\in\bB^1_+$ then there exists $\rho < 1$ such that $[x,0]\subset\bB^1(0,\rho)$.  We have $\bB^1(0,\rho)\subset\Int(\bB^1(0,\rho'))\subset\bB^1_+$ for any $\rho'\in(\rho,1)$, so that the connected analytic curve $\bB^1(0,\rho)$ in $\bB^1_+$ is overconvergent.

  Suppose now that the assertion is true for $\bB^{n-1}$.  Let $x\in\bB^n$, and define $\phi\colon\bB^n\to\bB^{n-1}$ and $\sigma\colon\bB^{n-1}\to\bB^n$ as in~\ref{art:paths.in.Bn}.  Let $y = \phi(x)$, and let $\Gamma_y$ be a connected analytic curve in $\bB^{n-1}$ containing $[y,0]$.  Since $\sigma$ is a morphism, Lemma~\ref{lem:analytic.curve.properties}(\ref{acp.image}) implies that $\sigma(\Gamma_y)$ is a connected analytic curve in $\bB^n$.  The fiber $\phi\inv(y)\cong\bB^1_{\sH(y)}$ is a connected analytic curve in $\bB^n$ by~\ref{art:paths.in.Y}, and $\phi\inv(y)\cup\sigma(\Gamma_y)$ is an analytic curve in $\bB^n$ by Lemma~\ref{lem:analytic.curve.properties}(\ref{acp.union}), which is connected because $\sigma(y)\in\phi\inv(y)\cap\sigma(\Gamma_y)$.  By definition $[x,0] = [x,\sigma(y)]\cup\sigma([y,0])$, which is contained in $\phi\inv(y)\cup\sigma(\Gamma_y)$ because $[x,\sigma(y)]\subset\phi\inv(y)$ and $[y,0]\subset\Gamma_y$.

  Now suppose that $x\in\bB^n_+$.  Then $y = \phi(x)\in\bB^{n-1}_+$, so by induction we can assume that $\Gamma_y$ is overconvergent and contained in $\bB^{n-1}_+$.  Since $x\in\bB^n_+$ we have $x\in(\bB^1_+)_{\sH(y)}\subset\bB^1_{\sH(y)} = \phi\inv(y)$, so as in the base case, we have that $[x,\sigma(y)]$ is contained in an overconvergent, connected analytic curve in $\bB^n_+$ of the form $\bB^1_{\sH(y)}(0,\rho)\subset\phi\inv(y)\cap\bB^n_+$ for $\rho < 1$.  Then $[x,0]$ is contained in $\bB^1_{\sH(y)}(0,\rho)\cup\sigma(\Gamma_y)$, which is a connected, overconvergent curve in $\bB^n_+$ by Lemma~\ref{lem:analytic.curve.properties}(\ref{acp.union},\ref{acp.image}).
\end{proof}

Our paths $[x,0]$ satisfy the following crucial locality lemma, which is analogous to \cite[Lemma~3.2.2(ii)]{berkovic90:analytic_geometry} in Berkovich's setting.

\begin{lem}\label{lem:path.locality}
  Let $y\in\bB^n$.  For any open neighborhood $U$ of $[y,0]\subset\bB^n$, there exists an open neighborhood $V$ of $y$ such that $[y',0]\subset U$ for all $y'\in V$.
\end{lem}

\begin{proof}
  Let $X = \bB^{n-1}$, let $Y = \bB^n = X\times\bB^1$, let $\phi\colon Y\to X$ be the projection, and let $\sigma\colon X \cong X\times\{0\}\inject Y$ be the section, as defined in~\ref{art:paths.in.Bn}.  By induction on $n$, it suffices to prove the following statement: for every $y\in Y$ and every open neighborhood $U$ of $[y,\sigma(\phi(y))]$, there is an open neighborhood $V$ of $y$ such that $[y',\sigma(\phi(y'))]\subset U$ for all $y'\in V$.

  We adapt the proof of~\cite[Lemma~3.2.2(ii)]{berkovic90:analytic_geometry} and fill in some details.  Let $\sA = k\angles{T_1,\ldots,T_{n-1}}$, so $X = \sM(\sA)$ and $Y = \sM(\sA\angles T)$.  Let $y\in Y$ and let $x = \phi(y)$.  We have $\phi\inv(x) = \sM(\sH(x)\angles T) = \bB^1_{\sH(x)}$.
  For the rest of this proof we use $0_x$ to denote the origin of this ball, i.e., $0_x=\sigma(x)$.
  The smallest closed ball in $\bB^1_{\sH(x)}$ containing $y$ and $0_x$ is $\bB^1(0_x,|T(y)|)$; let $m(y,0_x)$ be its maximal point.  We have the equalities
  \begin{equation}\label{eq:paths.as.sets}
    \begin{aligned}
      \begin{split}
        [y, m(y,0_x)] &= \bigl\{ z\in\phi\inv(x)\bigm|  |T(z)|\leq|T(y)| \text{ and } |f(z)|\geq|f(y)| \text{ for all }f\in\sA\angles T \bigr\} \\
        [m(y,0_x), 0_x] &= \bigl\{ z\in\phi\inv(x)\bigm| |T(z)|\leq|T(y)| \text{ and } |f(z)|\geq|f(0_x)| \text{ for all }f\in\sA\angles T \bigr\}
      \end{split}
    \end{aligned}
  \end{equation}
  as in the proof of~\cite[Lemma~3.2.2(ii)]{berkovic90:analytic_geometry}, and $[y,0_x] = [y,m(y,0_x)]\cup[m(y,0_x),0_x]$ by definition.

Let $U$ be an open neighborhood of $[y,0_x]$ in $Y$, so $U_x = U\cap\phi\inv(x)$ is an open neighborhood of $[y,0_x]$ in $\bB^1_{\sH(x)}$.  The ball $\bB^1_{\sH(x)}$ has a basis of so-called \defi{standard} open sets~\cite[\S4.2]{berkovic90:analytic_geometry} of the form $B_+\setminus\bigcup_{i=1}^m B_i$, where $B_+$ is an open ball in $\bB^1_{\sH(x)}$ with a rig-point as a center and positive radius (or is all of $\bB^1_{\sH(x)}$) and the $B_i$ are disjoint closed balls in $\bB^1_{\sH(x)}$ contained in $B_+$.  This is made explicit in~\ref{topology of unit disc}, where we also see that the union of two standard open subsets  is a standard open subset.
Since $[y,0_x]$ is a connected set contained in a union of finitely many standard open sets inside $U_x$, we see that $[y,0_x]$ has a standard open neighborhood $B_+\setminus\bigcup_{i=1}^m B_i$ contained in $U_x$.  Any open ball containing $0_x$ has the form $\bB^1_+(0_x,\rho) = \{z\mid|T(z)|<\rho\}$, so we have $B_+=\bB^1_+(0_x,\rho)$ for suitable $\rho>0$ (or $B_+=\bB^1_{\sH(x)}$). Since $B_+$ contains $m(y,0_x)$, we have $\rho>|T(y)|$.
Let $f_i\in\sH(x)[T]$ be the minimal polynomial of a rig-center of $B_i$.  By Lemma~\ref{ball described by minimal polynomial}, we have $B_i=\{z \in \bB_{\sH(x)}^1\mid|f_i(z)|\leq\rho_i\}$ for suitable $\rho_i>0$.
  Since $B_i$ does not contain $0_x$ or $y$, we have $\rho_i < \min\{|f_i(y)|,|f_i(0_x)|\}$.  By a density argument, we may assume $f_i\in\sA[T]$.  For $\epsilon >0$,  let
  \[
    \begin{split}
      U' &= \bigl\{ z\in Y\bigm||T(z)| < |T(y)|+\epsilon \text{ and } \forall i\,|f_i(z)| > \min\{|f_i(y)|,\,|f_i(0_x)|\}-\epsilon \bigr\} \\
      A' &= \bigl\{ z\in Y\bigm||T(z)| \leq |T(y)|+\epsilon \text{ and } \forall i\,|f_i(z)| \geq \min\{|f_i(y)|,\,|f_i(0_x)|\}-\epsilon\bigr\}.
    \end{split}
  \]
  Then $A'$ is a compact set containing $U'$.  Using \eqref{eq:paths.as.sets}, we deduce that  $U'\cap\phi\inv(x)$ is an open neighborhood of $[y,0_x]$ in $\phi\inv(x)$, which is contained in $B_+\setminus\bigcup_{i=1}^m B_i$ and hence in
  $U$ for sufficiently small $\epsilon$.  Decreasing $\epsilon$ if necessary, we may even assume  that $A'\cap\phi\inv(x) \subset U$.  Then the compact set $A'\cap(Y\setminus U)$ is disjoint from $\phi\inv(x)$, so there exists an open neighborhood $W$ of $x$ such that $U'\cap\phi\inv(W)\subset U$.  Shrinking $W$ if necessary, we may assume by continuity of $\sigma$ and using $0_x=\sigma(x)$ that
  \begin{equation}\label{eq:nbhd.of.section}
    |f_i(\sigma(x'))| > |f_i(0_x)| - \epsilon \quad\text{for all } x'\in W \text{ and all } i = 1,\ldots,m.
  \end{equation}
  Replacing $U$ by $U'\cap\phi\inv(W)$, we may assume
  \[
    U = \bigl\{ z\in Y\bigm||T(z)| < |T(y)|+\epsilon \text{ and } \forall i\,|f_i(z)| > \min\{|f_i(y)|,\,|f_i(0_x)|\}-\epsilon  \bigr\} \cap\phi\inv(W).
  \]

  Let $y'\in U$ and $x' = \phi(y')\in W$, and let $0_{x'} = \sigma(x')$.
  Using~\eqref{eq:paths.as.sets} with $y'$ in place of $y$, we have
  $[y',0_{x'}] \subset U$ because for $z\in[y',0_{x'}]$ we have $|T(z)|\leq|T(y')|<|T(y)|+\epsilon$, and $|f_i(z)|\geq \min\{|f_i(y')|,|f_i(0_{x'})|\}>\min\{|f_i(y)|,|f_i(0_x)|\}-\epsilon$ for each $i$; this uses~\eqref{eq:nbhd.of.section} for the second inequality. This proves the lemma.
\end{proof}

The proof of Theorem~\ref{thm:conn.by.curves} requires the following lemma, which is an immediate consequence of Lemma~\ref{lem:analytic.curve.properties}(\ref{acp.image},\ref{acp.scalars}).  We say that a \emph{subset $\Sigma\subset X$ is connected by curves} if for every pair of points $x,y\in\Sigma$, there exists a connected analytic curve $\Gamma$ in $X$ contained in $\Sigma$ such that $x,y\in\Gamma$, and likewise for overconvergent curves.

\begin{lem}\label{lem:img.of.curve.connected}
  Let $X$ be a $k$-analytic space, let $k'/k$ be an analytic field extension, let $Y$ be a $k'$-analytic space, and let $f\colon Y\to X_{k'}$ be a morphism.  If $Y$ is connected by curves then $\pi_{k'/k}\circ f(Y)$ is connected by curves, and likewise for overconvergent curves.
\end{lem}

\begin{proof}[Proof of Theorem~\ref{thm:conn.by.curves}]
  Let $X$ be a connected $k$-analytic space.  Define a relation $\sim$ on $X$ by $x\sim y$ if there exists a connected analytic curve $\Gamma$ in $X$ such that $x,y\in\Gamma$.   This is an equivalence relation by Lemma~\ref{lem:analytic.curve.properties}(\ref{acp.union}), so since $X$ is connected, it suffices to show that equivalence classes are open.  It is enough to prove that every point $x\in X$ admits a (not necessarily open) neighborhood that is connected by  curves.  This is a property that we can prove for any $k$-analytic space, so we no longer assume that $X$ is connected. By the definition of a $k$-analytic space in \cite[\S 1.2]{berkovic93:etale_cohomology}, the point $x$ has a neighborhood of the form $X_1\cup\cdots\cup X_n$, where each $X_i$ is an affinoid domain in $X$ that contains $x$.  If $U_i$ is a neighborhood of $x$ in $X_i$ then $\bigcup_{i=1}^n U_i$ is a neighborhood of $x$ in $X$, so it suffices to show that $x$ has a neighborhood in each $X_i$ that is connected by curves. We may thus assume that $X$ is affinoid.  There exists an analytic  extension field $k'/k$ such that $X_{k'}$ is strictly $k'$-affinoid and such that $\pi_{k'/k}\colon X_{k'}\to X$ admits a continuous section $\sigma\colon X\to X_{k'}$: see~\cite[Lemma 3.2.2(i)]{berkovic90:analytic_geometry}.  If $U\subset X_{k'}$ is a neighborhood of $\sigma(x)$ that is connected by curves then Lemma~\ref{lem:img.of.curve.connected} implies that $\pi_{k'/k}(U)$ is connected by curves; this is a neighborhood of $x$ because $\sigma\inv(U)\subset\pi_{k'/k}(U)$.  Hence we may assume that $X$ is strictly $k$-affinoid.  Then $X$ has finitely many irreducible components; as above, it is enough to show that $x$ has a neighborhood in each irreducible component that is connected by curves, so we may assume that $X$ is irreducible.

  Now we will prove that any irreducible, strictly $k$-affinoid space $X$ is connected by curves.  We will show that for every $x\in X$, there is a neighborhood $U$ of $x$ such that $x\sim x'$ for all $x'\in U$.  This is not the same thing as showing that $U$ is connected by curves, as the curve connecting $x$ to $x'$ may leave $U$.  However, it does imply that the $\sim$-equivalence classes are open, which is enough because $X$ is connected.

  By the Noether normalization lemma \cite[Corollary 6.1.2/2]{bosch_guntzer_remmert84:non_archimed_analysis}, \cite[Corollary 2.1.16]{berkovic90:analytic_geometry}, there exists a finite, surjective morphism $\phi\colon X\to\bB^n$ for $n=\dim(X)$.  For $x'\in X$ let $y' = \phi(x')$, and let $\Sigma_{x'}$ be the connected component of $\phi\inv([y',0])$ containing $x'$.  Berkovich shows in the proof of~\cite[Lemma~3.2.5]{berkovic90:analytic_geometry} (which applies in our situation by Lemma~\ref{lem:path.locality}) that there is an open neighborhood $U$ of $x$ such that for all $x'\in U$, the sets $\Sigma_x$ and $\Sigma_{x'}$ contain a common point of $\phi\inv(0)$.

  By Lemma~\ref{lem:paths.in.curves}, there are connected analytic curves $\Gamma_y,\Gamma_{y'}$ in $\bB^n$ such that $[y,0]\subset\Gamma_y$ and $[y',0]\subset\Gamma_{y'}$.  Let $\Gamma_x$ (resp.\ $\Gamma_{x'}$) be the connected component of $\phi\inv(\Gamma_y)$ (resp.\ $\phi\inv(\Gamma_{y'})$) containing $x$ (resp.\ $x'$).  These are connected analytic curves in $X$ by Lemma~\ref{lem:analytic.curve.properties}(\ref{acp.components},\ref{acp.pullback}).  We have $\Sigma_x\subset\Gamma_x$ and $\Sigma_{x'}\subset\Gamma_{x'}$, so that $\Gamma_x\cap\Gamma_{x'}\neq\emptyset$.  Hence $x$ and $x'$ are contained in the connected analytic curve $\Gamma_x\cup\Gamma_{x'}$ in $X$.  Since $x'\in U$ was arbitrary, this completes the proof that a connected $k$-analytic space is connected by curves.

  \smallskip
  Now suppose that $X$ is a connected and boundaryless $k$-analytic space.  Define an equivalence relation $\sim'$ on $X$ by $x\sim'y$ if there exists a connected, overconvergent analytic curve $\Gamma$ in $X$ such that $x,y\in\Gamma$.   As before, we will prove that every point $x\in X$ admits a neighborhood that is connected by overconvergent curves, since this implies that the $\sim'$-equivalence classes are open.  By the definition the relative interior (\cite[\S 1.5]{berkovic90:analytic_geometry}, \cite[\S 4.2.4]{temkin15:berkovich_intro}), any point $x\in X$ admits an affinoid neighborhood $Y$.  Note that $x \in\Int(Y)$.  Let $k'/k$ be an analytic extension field such that $Y_{k'}$ is strictly $k'$-analytic and $\pi_{k'/k}\colon Y_{k'}\to Y$ admits a continuous section $\sigma\colon Y\to Y_{k'}$, as before.  We have $\pi_{k'/k}\inv(\Int(Y))\subset\Int(Y_{k'})$ by~\artref{relative boundary}\eqref{boundary and base change}, so $\sigma(x)\in\Int(Y_{k'})$.  As in the previous case, it is enough to find a neighborhood of $\sigma(x)$ in $Y_{k'}$ that is connected by overconvergent curves.  Hence we may assume that $Y$ is strictly $k$-affinoid and $x\in\Int(Y)$. Let $k'/k$ be the completion of an algebraic closure of $k$.  Then $Y_{k'}\to Y$ is open by~\cite[Corollary~1.3.6]{berkovic90:analytic_geometry}, so it suffices to find a neighborhood of any preimage of $x$ in $\Int(Y_{k'})$ that is connected by overconvergent curves by Lemma~\ref{lem:img.of.curve.connected}.  We may therefore assume that $k$ is algebraically closed. Let $Y_1,\ldots,Y_n$ be the irreducible components of $Y$ containing $x$.  Then $x\in\Int(Y_i)$ for each $i$; it is enough to show that $x$ admits a neighborhood in each $Y_i$ that is connected by overconvergent curves, so we may assume that $Y$ is irreducible and $x\in\Int(Y)$. Let $V$ be the connected component of $x$ in $\Int(Y)$.

  We will prove that $V$ is connected by overconvergent curves.  We proceed with an argument similar to the previous case: that is, for any $x\in V$, we will show that $x$ and $x'$ are connected by an overconvergent curve in $V$ for all $x'$ in a neighborhood of $x$.  Choose a finite, surjective morphism $\phi\colon Y\to\bB^n$.  The interior of $\bB^n$ is the (disjoint) union of all residue balls $\bB^n_+(z,1)$  by~\cite[Lemme~6.5.1]{chambert_ducros12:forms_courants} (all residue balls have $k$-rational centers $z$ since $k=\bar k$), and $\phi\inv(\Int(\bB^n))=\Int(Y)$ by finiteness of $\varphi$ (use \artref{relative boundary}\eqref{finite and boundary} and \eqref{transitivity of boundaries}); thus, after recentering, we can assume that $\phi(V)\subset\bB^n_+$.  For $x'\in V$ let $y' = \phi(x')\in\bB^n_+$ and let $\Sigma_{x'}$ be the connected component of $\phi\inv([y',0])$ containing $x'$.  Since $[y',0]\subset\bB^n_+$ we have $\phi\inv([y',0])\subset\Int(Y)$, and since $\Sigma_{x'}$ is connected and $x'\in V$ we have $\Sigma_{x'}\subset V$.  As before, there exists an open neighborhood $U$ of $x$ in $V$ such that $\Sigma_{x'}\cap\Sigma_x\neq\emptyset$ for all $x'\in U$.

  By Lemma~\ref{lem:paths.in.curves}, for any $y'\in\bB^n_+$ there is a connected, overconvergent analytic curve $\Gamma_{y'}$ in $\bB^n_+$ containing $[y',0]$.  For $x'\in V$ mapping to $y'$, we let $\Gamma_{x'}$ be the connected component of $\phi\inv(\Gamma_{y'})$ containing $x'$, so that $\Gamma_{x'}\subset V$ for the same reason that $\Sigma_{x'}\subset V$.  By Lemma~\ref{lem:analytic.curve.properties}(\ref{acp.components},\ref{acp.pullback}), this $\Gamma_{x'}$ is again a connected, overconvergent analytic curve in $X$.  Clearly $\Sigma_{x'}\subset\Gamma_{x'}$, so for $x'\in U$ we have $\Gamma_{x'}\cap\Gamma_x\neq\emptyset$.  This completes the proof as in the previous case.
\end{proof}

\section{The maximum principles for classically psh functions} \label{section: maximum principle for classically functions}

In this section, we consider a non-Archimedean field $k$. We will prove first that the local maximum principle holds for classically psh functions, and then we will show a global maximum principle in the affinoid case; more specifically, that that a classically psh function takes its global maximum at a Shilov point.

\begin{thm}[Local Maximum Principle] \label{thm: maximum principle for classically psh functions}
	Let $X$ be a $k$-analytic space   and let $u$ be a classically psh  function on $X$. If $u$ has a local maximum at $x \in \Int(X)$, then $u$ is locally constant at $x$.
\end{thm}

\begin{proof}
  We may replace $X$ by a connected, Hausdorff open neighborhood of $x$ in $\Int(X)$ to assume that $X$ is connected and boundaryless and that $u$ attains a global maximum at~$x$.  Let $R = u(x)$.  Let $\Gamma$ be a connected, overconvergent analytic curve  in $X$ containing $x$.  We claim that $u\equiv R$ on $\Gamma$.  Consider the set $\Gamma_{u=R} = \{z\in\Gamma\mid u(z)=R\}$.  This set is closed in $\Gamma$ by upper semicontinuity of $u$, so since $\Gamma$ is connected, it suffices to show that $\Gamma_{u=R}$ is open in~$\Gamma$.

By~\artref{art:larger.extension}, there exist an analytic  extension field $k'/k$, a compact, separated strictly $k'$-analytic space $C$ of dimension at most $1$, a separated strictly $k'$-analytic space $C'$ of dimension at most $1$, and morphisms $\psi\colon C\to C'$ and $\phi'\colon C'\to X_{k'}$, such that $\psi(C)\subset\Int(C')$ and $\pi_{k'/k}\circ\phi(C) = \Gamma$, where $\phi = \phi'\circ\psi$.  The composition $u\circ\pi_{k'/k}\circ\phi'$ is subharmonic on $1$-dimensional connected components of $C'$ because $u$ is classically psh.

Let $y\in\Gamma_{u=R}$, let $z\in C$ be a point lying over $y$, and let $z' = \psi(z)\in\Int(C')$, so that $R = u\circ\pi_{k'/k}\circ\phi'(z') = u\circ\pi_{k'/k}\circ\phi(z)$.  If $\{z'\}$ is a connected component of $C'$ then clearly $u\circ\pi_{k'/k}\circ\phi'\equiv R$ in a neighborhood of $z'$.  Otherwise, the connected component of $z'$ in $C'$ is a strictly $k'$-analytic curve, in which case $u\circ\pi_{k'/k}\circ\phi'\equiv R$ in a neighborhood of $z'$ by Proposition~\ref{prop:subharmonic.properties}(\ref{shp.maximum}) (here we use $z'\in\Int(C')$).  Taking the $\psi$-inverse image of such a neighborhood, we see that $u\circ\pi_{k'/k}\circ\phi\equiv R$ in a neighborhood of $z$.  Applying this argument to each $z\in(\pi_{k'/k}\circ\phi)\inv(y)$, we construct an open set $V\subset C$ containing $(\pi_{k'/k}\circ\phi)\inv(y)$ such that $u\circ\pi_{k'/k}\circ\phi|_V\equiv R$.

The complement $C\setminus V$ is a closed subset of the compact space $C$, so it is a compact subset of $C$ that is disjoint from $(\pi_{k'/k} \circ \phi)\inv(y)$.  Its image is a compact subset of $X$ that does not meet $y$, so there is an open neighborhood $W$ of $y$ such that $(\pi_{k'/k}\circ\phi)\inv(W)\subset V$.  It follows that $u\equiv R$ on $W\cap\Gamma$.  This proves that $\Gamma_{u=R}$ is open in~$\Gamma$. As remarked above, this yields $u\equiv R$ on $\Gamma$.

Let $y\in X$ be any point.  By Theorem~\ref{thm:conn.by.curves}, there is a connected, overconvergent analytic curve $\Gamma$ in $X$ containing $x$ and $y$.  We showed above that $u\equiv R$ on $\Gamma$, so that $u(y)=R$.  As $y$ was arbitrary, we have $u\equiv R$ on $X$.
\end{proof}

The global maximum principle for classically psh functions involves the Shilov boundary. For this, we need the following two lemmas.  The first is well-known, but we were unable to find a reference.

\begin{lem} \label{Shilov boundary vs usual boundary}
  Let $X$ be a $k$-affinoid space of pure dimension $d \geq 1$. Then the Shilov boundary of $X$ is contained in the relative boundary $\partial X$ of $X$ over $k$.  In particular, $\del X \neq \emptyset$.
\end{lem}

\begin{proof}
  Let $k'/k$ be an analytic field extension such that $X_{k'}$ is strictly affinoid. Note that the base change morphism $\pi_{k'/k}$ maps  $\partial X_{k'}$ to  $\partial X$~\artref{art:relative.interior}\ref{boundary and base change'} and maps the Shilov boundary onto the Shilov boundary \cite[Proposition 2.21]{gubler_rabinoff_werner:tropical_skeletons}, so we may assume that $X$ is strictly affinoid and $k'$ non-trivially valued.
	Choose a formal affine $k^\circ$-model $\fX$ of $X$. The interior $X \setminus \partial X$ consists of the points $x \in X$ such that the closure of $\red_\fX(x)$ in $\fX_s$ is proper over the residue field $\td k$ by \cite[Lemme 6.5.1]{chambert_ducros12:forms_courants}. Since $\fX$ is affine, the special fiber $\fX_s$ is affine, so $x \in \partial X$ if and only if $\red_\fX(x)$ is not a closed point of $\fX_s$. Since $d \geq 1$,  the special fiber $\fX_s$ is of pure dimension $d \geq 1$, and hence the generic points of $\fX_s$ are not closed. By \cite[Proposition A.3]{gubler_martin19:zhangs_metrics}, the divisorial points of $\fX$ (i.e.~the points of $X$ which reduce to the generic points of $\fX_s$) are precisely the Shilov points of $X$. We conclude that all Shilov points are in the boundary of $X$.
\end{proof}

Let $X=\sM(\sA)$ be {an affinoid} space over $k$ of {pure} dimension $d$. By a result of Ducros \cite[Lemme~3.1]{ducros12:squelettes_modeles}, the boundary $\partial X$ of $X$ can be expressed as a union of {affinoid} spaces of lower dimension, defined over certain extension fields of $k$. We use his argument to give a stratification of $\partial X$.  Choose non-nilpotent elements $g_1, \dots, g_r \in \sA$ such that the graded reductions $\td g_1^\bullet, \dots, \td g_r^\bullet$ generate the graded $\td k^\bullet$-algebra $\td{\sA}^\bullet$.
Let $r_i$ be the spectral radius of $g_i$.
Then for any $I \subset \{1,\dots,r\}$, we get a morphism
$$p_I  \colon X \longrightarrow \bB(I),\qquad x \mapsto (g_i(x))_{i \in I}$$
to the closed  poly-disc $\bB(I)=\prod_{i \in I}\bB(0,r_i)$ of dimension $s=|I|$. Let $\eta^{(I)}$ be the weighted Gauss point of $\bB(I)$ and let $k_I \coloneqq \sH(\eta^{(I)})$, which is the completion of the field of rational functions $k((t_i)_{i \in I})$ in the variables $(t_i)_{i \in I}$ with respect to the weighted Gauss valuation for the weight $(r_i)_{i \in I}$.

\begin{lem} \label{stratification of the boundary'}
	With the above notations, set $X_I \coloneqq p_I^{-1}(\eta^{(I)})$. By convention, we set $X_\emptyset=X$.  Then:
	\begin{enumerate}
		\item \label{stratification 1'}
                  $X_I$ is a  $k_I$-affinoid space of pure dimension ${d-s}$ if non-empty.
		\item \label{stratification 2'}
		$\partial X_I = \bigcup_{i \in \{1,\dots,r\} \setminus I} X_{I \cup \{i\}}.$
		\item \label{stratification 3'}
		$X_I=\{x \in X \mid \text{$(\td g_i^\bullet(x))_{i \in I}$ are algebraically independent in $\td \sH^\bullet(x)$ over $\td k^\bullet$}\}.$
		\item \label{stratification 4'}
		If $X$ has pure dimension $d$, then the Shilov boundary of $X$ is equal to $\bigcup_{|I|=d}X_I$.
	\end{enumerate}
\end{lem}

We note that the affinoid space $X_I$ is naturally homeomorphic to the fiber $p_I^{-1}(\eta^{(s)})\subset X$ in the subspace topology: see~\cite[\S 1.4]{berkovic93:etale_cohomology}. If $I$ is non-empty, then we call $X_I$ a \defi{stratum set} for the boundary.

\begin{proof}
	Since $p_I$ is a morphism between the $k$-affinoid spaces $X$ and $\bB(I)$ of pure dimension $d$ and $s$, respectively, and since the Gauss point $\eta^{(I)}$ is an Abhyankar point of $\bB(I)$, we deduce from  \cite[Lemma 1.5.11]{ducros18:families} that $X_I$ is a $\sH(\eta^{(I)})$-affinoid space of pure dimension $d-s$ if non-empty. This proves \eqref{stratification 1'}.

	The remaining argument is based on the following description of the boundary given by Ducros in \cite[Lemme 3.1]{ducros12:squelettes_modeles}. For any $x \in X=\sM(\sA)$, we get an induced character $\psi\colon \sA \to \sH(x)$. Then Temkin's graded reduction of the germ $(X,x)$ is given by~\cite[\S 2]{temkin00:local_properties}
	$$\red(X,x)= \P_{\td\sH(x)/\td k}\{\td\psi^\bullet(\sA)\}.$$
	By definition, this is the set of graded valuations $w$ on $\td\sH^\bullet(x)$ which are trivial on $\td k$ and which satisfy $|\td\psi^\bullet(g)|_w\leq 1$ for all $g \in \sA$. Using that the graded reductions of our fixed family  $g_1, \dots, g_r \in \sA$ generate the graded $\td k^\bullet$-algebra $\td{\sA}^\bullet$, it is enough to check  $|\td\psi^\bullet(g)|_w\leq 1$ for $g=g_1, \dots, g_r$.
	Based on Temkin's result on graded reductions, Ducros shows that $x \in \partial X$ is equivalent to the existence of some $i \in \{1,\dots,r\}$ and of a graded valuation $w$ on  $\td\sH^\bullet(x)$ which is trivial on $\td k^\bullet$ such that $|\psi(g_i)|_w>1$. Ducros shows in his argument that this is equivalent to $\td g_i^\bullet(x) \in \td\sH^\bullet(x)$ being transcendental over $\td k^\bullet$. For $i=1, \dots, r$, he considers the morphism $p_i\colon X \to \bB(0,r_i)$ induced by $g_i$.
	Let $X_i$ be the fiber over  the weighted Gauss point $\eta^{(i)}$ of the closed disc $\bB(0,r_i)=\sM(k\langle r_i^{-1}t \rangle)$ (of radius $r_i$).
	{By \eqref{stratification 1'}, we have  that $X_i$ is a $\sH(\eta^{(i)})$-affinoid space of pure dimension $d-1$ if non-empty.} It follows from the above that the union of all of the $X_i$ is equal to $\partial X$.

	Let us summarize the construction.
	For each $i\in \{1,\dots,r\}$, we define
	$$X_i \coloneqq \{x \in X \mid \text{$\td g_i^\bullet(x)$ is transcendental over $\td  k^\bullet$}\}.$$
	This set might be empty. We have seen that $g_i$ defines a morphism $p_i\colon X \to \bB(0,r_i)=\sM(k\langle r_i^{-1}t \rangle)$ for the Tate algebra $k\langle r_i^{-1}t \rangle$ in the variable $t$, and that $X_i$ is the fiber of $p_i$ over the weighted Gauss point $\eta^{(i)}$. In any case, $X_i$ is an affinoid space over the non-archimedean field  $k_i \coloneqq \sH(\eta^{(i)})$ which is the completion of $k(t)$ with respect to the weighted Gauss valuation for the weight $r_i$. If $X_i$ is non-empty, then $X_i$ has pure dimension $ d-1$. Ducros has shown that
	$$\partial X = \bigcup_{i \in \{1,\dots,r\} } X_{i}.$$

	We prove now the claims \eqref{stratification 2'}--\eqref{stratification 3'} by induction on $s=|I|$. If $s=0$, then $I=\emptyset$ and $X_I=X$. Then the above description of the boundary due to Ducros proves \eqref{stratification 2'} and \eqref{stratification 3'}. Now assume that $s>0$. We pick $i \in I$ and set $J \coloneqq I \setminus \{i\}$. As seen above, we have that $X_i\coloneqq p_i^{-1}(\eta^{(i)})$ is an   affinoid space over $k_i=\mathscr H(\eta^{(i)})$, and we denote the corresponding affinoid algebra by $\sA_i\coloneqq \sO(X_i)$.
	Note here that $k_i$ is the completion of $k(t)$ with respect to the weighted Gauss norm for the weight $r_i$.
	For $\sA_i=\sA \hat\otimes_k k_i$, the same elements $g_1, \dots, g_r$ induce generators of $\td{\sA}_i^\bullet$ as a graded $\td k_i^\bullet$-algebra. Note that $\mathscr H(\eta^{(I)})$ is the completion of $k((t_{i'})_{i' \in I})$ with respect to the Gauss norm for the weight $(r_{i'})_{i' \in I}$  and hence also the completion of $k_i(r_j)_{j \in J}$ with respect to the weighted Gauss norm for the weight $(r_j)_{j \in J}$ identifying the variable $t\in k_i$ with $t_i$. It follows readily that for any $H \subset \{1,\dots, r\}$, we have $(X_i)_H=X_{\{i\} \cup H}$ as  affinoid spaces over $k_I$. Applying this with $H=J$ and using our induction hypothesis for $X_i$,
	property \eqref{stratification 2'} for the boundary of $X_I=(X_i)_J$ follows by induction applied to $X_i$. The same inductive argument shows that $X_I$ is the set of points $x \in X_1$ where $(\td g_j(x))_{j \in J}$ are {algebraically independent} over $\td k_i$. We have seen that $X_i$ is non-empty if and only if $\td g_i(x)$ is transcendental over $\td k$. Since  $t=t_i$ in $k_i$ satisfies $p_i^*(t)=g_i$, we conclude that $x \in X_I$ if and only if $(\td g_l(x))_{l \in I}$ are {algebraically independent} over $\td k$. This proves \eqref{stratification 3'}.

	Finally, we claim  that the union of the  minimal strata $X_I$ for $|I|=d$ is the Shilov boundary of $X$. If $d=0$, this is obious as $X=X_\emptyset$. Now assume that $d \geq 1$.   If $x\in X_I$ then  \eqref{stratification 3'} shows that the graded residue field $\td\sH(x)^\bullet$ contains
	{the $d$ algebraically independent elements $(\td g_i)_{i \in I}$ over $\td k^\bullet$}. Since the reduction $\td X$ in the sense of Temkin \cite[Section 3]{temkin04:local_properties_II} is of pure dimension $d$ (using \cite[Proposition 3.1(v)]{temkin04:local_properties_II} for base change to the strictly affinoid case), we conclude that the reduction of $x$ in $\td X$ is a generic point of $\td X$. By \cite[Proposition 3.3]{temkin04:local_properties_II}, we conclude that $x$ is a Shilov point of $X$.
	Conversely, for every $a \in \sA$, the function $|a|$ takes its maximum in a Shilov point.  We conclude from Lemma~\ref{Shilov boundary vs usual boundary} that $|a|$ takes its maximum in $\partial X\neq\emptyset$.   By~\eqref{stratification 2'}, we conclude that $|a|$ takes its maximum in one of the {strata sets} $X_I$.
	By induction on the dimension $d$, this occurs when $|I|=d$ proving \eqref{stratification 4'}.
\end{proof}

We return to the setting where $k$ is any non-Archimedean field which might be trivially valued. We give now the \emph{global maximum principle for affinoids.}

\begin{thm} \label{thm: global affinoid maximum principle}
	Let $h$ be a classically psh function on a $k$-affinoid space $X$. Then $h$ takes its maximum in a Shilov point of $X$.
\end{thm}

An analogue for $\theta$-psh functions on a proper scheme is in \cite[Proposition 4.22]{gubler_martin19:zhangs_metrics}.

\begin{proof}
  By base change to a non-trivially valued analytic extension field $k'/k$ such that $X_{k'}$ is strictly affinoid, and using that $\pi_{k'/k}$ maps the Shilov boundary of $X_{k'}$ to the Shilov boundary of $X$~\cite[Proposition 2.21]{gubler_rabinoff_werner:tropical_skeletons}, we may assume that $X$ is strictly affinoid and that $k$ is non-trivially valued.
	Let $X= \bigcup_j X_j$ be the partition of $X$ into its irreducible components.
	It follows from \cite[Proposition 2.15]{gubler_rabinoff_werner:tropical_skeletons} that the union of the Shilov points of all $X_j$ gives the Shilov points of $X$. So it is enough to prove the claim for every $X_j$, and hence we may assume $X$ irreducible. This will be a bit to strong for our inductive argument later, but the argument shows that we may assume $X$ of pure dimension $d $.

	If $d \geq 1$, then we claim that $h$ takes its maximum in the boundary $\partial X$.
	We prove this by contradiction.
	It is clear that the boundary $\partial X$ is the disjoint union of the boundaries of the connected components of $X$, and the same holds if we replace the boundary by the Shilov boundary, so we may assume $X$ connected.
	Let $M$ be the set of points of $X$ where $h$ is maximal. We assumed that $M \cap\partial X = \emptyset$. Since $h$ is continuous, it is clear that $M$ is a closed subset of $X$.
	We know from Theorem \ref{thm: maximum principle for classically psh functions} that if a classically psh function takes a (local) maximum on $X \setminus \partial X$, then $h$ is locally constant. This implies $M$ open in $X$.  As $M$ is non-empty and $X$ is connected, we deduce $M=X$. Since $\partial X \neq \emptyset$ as we have $d \geq 1$ (see Lemma \ref{Shilov boundary vs usual boundary}), we get $M \cap \partial X \neq \emptyset$, which is a contradiction.

	Let now $h$ be a  classically psh function on $X$. We need the description of the boundary $\partial X=\bigcup X_i$ given in Lemma \ref{stratification of the boundary'}, where $i$ runs through $\{1,\dots,r\}$ and $X_i=X_{\{i\}}$ as before. We claim by induction on $d$ that $h$ takes its maximum in one of the  sets $X_I$ with $|I|=d$, and hence in the Shilov boundary of $X$, which proves the theorem. If $h$ is constant, then this is obvious; this also settles the case $d=0$. If $h$ is non-constant, then we have seen at the beginning of the proof that $h$ takes its maximum in $\partial X$, and hence in one of the $X_i$. Since $X_i$ is of dimension $d-1$, we are done by induction as the strata sets of $\partial X_i$  are strata sets of $\partial X$ using Lemma \ref{stratification of the boundary'}.
\end{proof}

The following result was suggested by a referee of the paper \cite{GR25semipositivePL}. We are very grateful for this hint.

\begin{thm} \label{classically psh and pointwise limits}
Let $(u_j)_{j \in J}$ be a net of classically psh functions on the $k$-analytic space $X$ which converges pointwise to an upper semicontinuous function $u$. Then $u$ is classically psh.
\end{thm}

For a sequence of semipositive $\R$-PL functions converging pointwise to an $\R$-PL function, the above result was shown in \cite[Theorem 5.6]{GR25semipositivePL}. When $X$ is a proper algebraic variety, this was known before (\cite[Proposition 7.2]{gubler_kuenneman19:positivity} for uniform limits, \cite[Theorem 1.3]{gubler_martin19:zhangs_metrics} for pointwise limits).

\begin{proof}
  Since $u$ is upper semicontinuous, using the definition of classically psh functions, we may assume that $X$ is a connected smooth curve over a non-trivially valued field. We may assume that $u$ is not identically $-\infty$. Then for every strictly $k$-affinoid domain $U$ of $X$ and every harmonic function $h$ on $U$ with $u \leq h$ on $\partial U$, we have to show that $u \leq h$ holds on $U$.  Since $u_j \to u$ pointwise and since $\del U$ is finite, for any $\epsilon > 0$ we have $u_j - h \le \epsilon$ on $\del U$ for all $j$ sufficiently large.  Since $u_j-h$ is a classically psh function by Proposition \ref{prop:subharmonic.properties}(\ref{shp.cone}) and since the Shilov boundary of $U$ agrees with $\partial U$ \cite[\S 2.1.2]{thuillier05:thesis}, the global affinoid maximum principle in Theorem \ref{thm: global affinoid maximum principle} shows that $u_j-h \le \epsilon$ on $U$ for all $j$ sufficiently large.  It follows that $u - h \le \epsilon$ on $U$, which implies $u \leq h$ on $U$ since $\epsilon$ was arbitrary.
\end{proof}

\section{Pluriharmonic functions} \label{section: pluriharmonic functions}

In this section, we consider a $k$-analytic space $X$ over any non-Archimedean field $k$. Following the discussion of harmonic functions on curves in  \artref{art:harmonic.is.pm.subharmonic}, we generalize this notion as follows to higher dimensions.

\begin{defn} \label{def: pluriharmonic}
A function $h\colon X \to \R$ is called \emph{pluriharmonic} if $h$ and $-h$ are classically psh functions (see Definition \ref{def:classically.psh}).
\end{defn}

\begin{prop}\label{prop: pluriharmonic properties}
	Pluriharmonic functions are continuous and have the  properties:
	\begin{enumerate}
		\item\label{pluriharmonic.sheafiness} \textup{(Sheafiness)} The pluriharmonic functions form a sheaf on $X$.
                \item\label{pluriharmonic.analytic} If $f\in\Gamma(X,\sO_X^\times)$ then $\log|f|$ is pluriharmonic.
		\item\label{pluriharmonic.vector space} \textup{(Vector Space)} The pluriharmonic functions on $X$ form a real vector space which we denote by $\PH(X)$.
		\item\label{pluriharmonic.limits} \textup{(Limits)} If a function $h\colon X \to \R$ is locally a uniform limit of pluriharmonic functions, then $h$ is pluriharmonic.
		\item\label{pluriharmonic.maximum} \textup{(Maximum Principle)} If
		$h \in \PH(X)$ and $h$ attains a local extremum at $x\in\Int(X)$, then $h$ is constant in a neighborhood of $x$.
		\item\label{pluriharmonic.scalars} \textup{(Extension of Scalars)} Let $k'/k$ be an analytic extension field.  Then $h$ pluriharmonic if and only if $h\circ\pi_{k'/k}\colon X_{k'}\to\R$ is pluriharmonic.
		\item\label{pluriharmonic.functoriality} \textup{(Functoriality)} Let $Y$ be a $k$-analytic space and let $f\colon Y\to X$ be a morphism.  If
		$h \in \PH(X)$
		then $h\circ f \in \PH(Y)$, and the converse holds if $f$ is finite and surjective.
	\end{enumerate}
\end{prop}

\begin{proof}
  It is clear that a pluriharmonic function $h$ is continuous as $h$ and $-h$ are upper semicontinuous. Properties \eqref{pluriharmonic.sheafiness}, \eqref{pluriharmonic.analytic},      \eqref{pluriharmonic.vector space}, \eqref{pluriharmonic.scalars} and \eqref{pluriharmonic.functoriality} follow immediately from the corresponding properties of classically psh functions given in Proposition \ref{prop:classically.psh.properties}. In \eqref{pluriharmonic.limits}, we note that a uniform limit of pluriharmonic functions can be easily written as a decreasing (resp.~increasing) limit of pluriharmonic functions and hence \eqref{pluriharmonic.limits} follows from Proposition \ref{prop:classically.psh.properties}\eqref{psh.limits}.
The maximum principle \eqref{pluriharmonic.maximum} follows from the maximum principle for classically psh functions shown in Theorem \ref{thm: maximum principle for classically psh functions}.
\end{proof}

From now on, we denote by $\PH(X)$ the space of pluriharmonic functions on $X$. The affinoid global maximum principle for classically psh functions has the following consequence.

\begin{prop}  \label{pluriharmonic functions and affinoid maximum principle}
	Let $X$ be a  $k$-affinoid space with Shilov boundary $\Gamma$. Then every pluriharmonic function takes its global maximum and its global minimum in a Shilov point of $X$, and the $\R$-linear map
	$$\PH(X) \longrightarrow \R^\Gamma, \quad h \mapsto (h(x))_{x \in \Gamma}$$
	is injective.
\end{prop}

\begin{proof}
	By definition, the function $h \colon X \to \R$ is pluriharmonic if and only if $h$ and $-h$ are classically psh. It follows from Theorem \ref{thm: global affinoid maximum principle} that a pluriharmonic function takes its maximum and its minimum on the Shilov boundary $\Gamma$. In particular, a pluriharmonic function is zero if and only if its restriction to $\Gamma$ is zero, which proves the last claim.
\end{proof}

\begin{thm} \label{thm: finiteness of pluriharmonic}
	Let $X$ be a quasicompact $k$-analytic space. Then $\PH(X)$ is a finite dimensional real vector space.
\end{thm}

\begin{proof}
	Using quasicompactness, we can cover $X$ by finitely  affinoid  subsets $U_i$. Using Proposition \ref{pluriharmonic functions and affinoid maximum principle}, we know that $\PH(U_i)$ is finite dimensional for every $i$. Since pluriharmonic functions are determined by their restrictions to the $U_i$'s, this proves the claim. For later purposes, we note that the linear map
	\begin{equation} \label{injective map and Shilov boundary}
	\PH(X) \longrightarrow \R^\Gamma, \quad h \mapsto (h(x))_{x \in \Gamma}
	\end{equation}
is injective where $\Gamma$ is the union of the Shilov boundaries of all $U_i$'s.
\end{proof}

\begin{rem} \label{upper bound for pluriharmonic dimension}
	Assume that $k$ is non-trivially valued and that $X$  is a quasicompact $k$-analytic space.  Let $\fX$ be any formal model of $X$ over the valuation ring $k^\circ$. Then $\dim(\PH(X))$ is bounded by the number of divisorial points of $X$ associated to $\fX$.

	Indeed, such a formal model is quasi-compact, hence can be covered by finitely many formal affine open subsets $\fU_i$. Then the above argument applies to the strictly affinoid covering $U_i= \fU_{i,\eta}$ of $X$. Let $\Gamma$ be the set of divisorial points associated to $\fX$. This is just the union of the Shilov points of the $U_i$'s \cite[Proposition A.3]{gubler_martin19:zhangs_metrics}. Injectivity of \eqref{injective map and Shilov boundary} leads to the desired bound.
\end{rem}

\begin{thm} \label{pointwise convergence of pluriharmonic}
	Let $X$ be a $k$-analytic space over any non-Archimedean field $k$, and let $f \colon X \to \R$ be a continuous function. If $f$ is the pointwise limit of a sequence of pluriharmonic  functions on $X$, then the convergence is locally uniform and $f$ is a pluriharmonic  function.
\end{thm}

\begin{proof}
	This is a local statement, so we may assume that $X$ is quasicompact. Using the notation from the proof of Theorem \ref{thm: finiteness of pluriharmonic}, injectivity of \eqref{injective map and Shilov boundary} shows that
	$$\|h\| = \sup_{x\in \Gamma} |h(x)|$$
	defines a norm on the finite dimensional $\R$-vector space $\PH(X)$. Let $h_m$ be a sequence of pluriharmonic  functions converging pointwise to $f$. Then $h_m$ converges to some $h \in \PH(X)$ with respect to the norm $\metr$ as the sequence $(h_m(x))_{x \in \Gamma}$ converges in $\R^\Gamma$ and hence in every subspace. It follows from Proposition  \ref{pluriharmonic functions and affinoid maximum principle} that a pluriharmonic function takes its maximum and its minimum on  $\Gamma$ and hence $\metr$ is the sup-norm on $X$. We conclude that $h_m$ converges uniformly to $h$ on $X$ and hence $h=f$, proving the claim.
\end{proof}

\begin{art} \label{strictly semistable schemes}
  In the following, we assume $k$ non-trivially valued. A locally finitely presented quasicompact formal scheme $\fX$ over $k$ is called \defi{strictly semistable} if every $x \in \fX$ has an open neighborhood $\fU$ admitting an \'etale morphism
	\begin{equation}
		\label{eq:formal.sss.pair}
		\psi\colon \fU \To
		\Spf\big( k^\circ\angles{x_0, \dots, x_d} / \angles{x_0 \dots x_r - \pi} \big)
	\end{equation}
for some $r \leq d$ and some $\pi\ne0$ in the valuation ring $k^\circ$. Note that strictly semistable formal schemes are strictly polystable in the sense of \cite[Definition 1.2]{berkovic99:locally_contractible_I} and that the generic fiber $X$ is smooth. Berkovich has shown that there is a canonical skeleton $S(\fX)$ in $X$ associated to $\fX$ coming with a canonical deformation retraction $\tau\colon X \to S(\fX)$ \cite[Theorem 5.2]{berkovic99:locally_contractible_I}. The skeleton of $\fU$ is the simplex
$$S(\fU)=\bigl\{u \in \R_{\geq 0}^{r+1} \mid u_0+\dots + u_r =v(\pi)\bigr\}$$
which is isomorphic to the standard simplex $\{y \in \R_{\geq 0}^r \mid y_1+\dots+y_r \leq v(\pi)\}$ of length $v(\pi)$, where $v$ is the valuation of $k$. The skeleton of $\fX$ is defined by gluing the skeletons $S(\fU)$ along a formal open covering of $\fX$ by $\fU$'s as above, which gives $S(\fU)$ a canonical triangulation.
\end{art}

\begin{prop} \label{semistable and pluriharmonic}
  Let $X$ be a good, quasicompact, strictly $k$-analytic space that has a strictly semistable formal $k^\circ$-model. Then every pluriharmonic function $h$ on $X$ is $\R$-PL.
\end{prop}

\begin{proof}
  Let $\fX$ be a strictly semistable formal $k^\circ$-model of $X$. Since $\R$-piecewise linearity is a $\G$-local property with respect to the Grothendieck topology generated by the strictly analytic domains, we may assume that $\fX$ is an open subset $\fU \to \Spf\big( k^\circ\angles{x_0, \dots, x_d} / \angles{x_0 \dots x_r - \pi} \big)$ as in~\artref{strictly semistable schemes}, and that $\fX$ is formal affine. The Shilov points of $X$ are then the vertices of the simplex $S(\fX)$. Using that $S(\fX)$ is a simplex, there is a unique linear function $\ell\colon S(\fX) \to \R$ that agrees with $h$ on the vertices.  Since $\tau$ can be seen as a tropicalization (see \cite[4.3]{gubler_rabinoff_werner16:skeleton_tropical} for the argument), it is clear that $\ell\circ \tau$ is $\R$-PL, and since $\ell\circ\tau$ is a linear combination of pluriharmonic functions (namely, the functions $-\log|x_i|$ for $i=0,\ldots,d$), we see that $\ell\circ\tau$ is pluriharmonic by Proposition~\ref{prop: pluriharmonic properties}\eqref{pluriharmonic.analytic}.  By Proposition \ref{pluriharmonic functions and affinoid maximum principle}, we have $h=\ell \circ \tau$, proving the claim.
\end{proof}

\appendix

\section{The topology of the unit disk} \label{sec: topology of unit disc}

Let $k$ be a non-archimedean field which might be trivially valued.  In this appendix, we give an elementary description of the topology of the unit disc $\bB_k^1=\sM(k\langle T \rangle)$.  These are well-known facts which we present in the way we need in the body of text.

\begin{art} \label{definition disks}
A rig-point of $\bB^1_k$ is induced by a maximal ideal of $k\langle T \rangle$. Let $k'$ be the completion of an algebraic closure $k^a$ of $k$ and let $G$ be the Galois group of $k^a/k$. By \cite[1.3.5]{berkovic90:analytic_geometry} or \cite[(3.1.1.2)]{ducros14:structur_des_courbes_analytiq}, we may write $\bB_k^1=\bB_{k'}^1/G$ as a topological quotient; we denote the quotient map by $\pi\colon \bB_{k'}^1 \to \bB_k^1$. For a rig-point $z$ and $0<\rho \leq 1$, we define \emph{the closed  disc
	$\bB_k^1(z,\rho)$ with center $z$ and radius $\rho$}  by taking any $z' \in \pi^{-1}(z)$ and then setting
$$\bB_k^1(z,\rho) \coloneqq \pi\bigl(\{x' \in \bB_{k'}^1 \mid |(T-z')(x')|\leq \rho\}\bigr).$$
This is a closed set in $\bB_k^1$ which does not depend on the choice of $z'$ and which contains the center $z$. We often simply write $\bB^1(z,\rho)=\bB_k^1(z,\rho)$.  Moreover, we can choose any rig-point of $\bB^1(z,\rho)$ as a center. Note that for any rig-point $y'$ of $\bB_{k'}^1$, the image $\pi(\bB_{k'}^1(y',\rho))$ is a closed disc $\bB_k(z,\rho)$ for a suitable rig-point $z$ as center. In the trivially valued case, we can simply take $z=\pi(y')$ and in the non-trivially valued case, this follows by density of rigid points \cite[Proposition 2.1.15]{berkovic90:analytic_geometry} in a strictly affinoid space. Similarly, we define \emph{the open  disc
	$\bB_+^1(z,\rho)$ with center $z$ and radius $\rho$} using $ |(T-z')(x')|< \rho$ as a defining inequality over $k'$. The above remarks also apply to open balls. It follows from the ultrametric triangle inequality that two open or closed balls are either disjoint or one is contained in the other.
\end{art}

\begin{art} \label{topology of unit disc}
The topology of the unit disc $\bB_k^1$ is well-known. Roughly speaking, it has the structure of an infinite tree with all branches connected to the Gauss point $\eta_1$ at the top (see \cite[Example 1.4.4, \S 4.2]{berkovic90:analytic_geometry} or \cite[\S 3.4]{ducros14:structur_des_courbes_analytiq} for more details). We call $U \subset \bB_k^1$ a \defi{standard open subset} if it has the form $U=B_+ \setminus \bigcup_i B_i$ where $B_+$ is either an open ball $\bB_+(z,\rho)$ with a rig-point $z$ as center and with $\rho \in (0,1)$ or $B_+=\bB_k^1$, and where the $B_i$ are finitely many closed balls with rig-points as centers and with positive radii.  Since two discs are either disjoint or one is contained in the other, we see that the union and the intersection of two standard open subsets are again  standard open subsets.  We claim that the standard open subsets form a basis of topology for $\bB_k^1$. In the non-trivially valued case, this follows from \cite[\S 4.2]{berkovic90:analytic_geometry} using \S \ref{definition disks} to find rig-centers in the balls. In the trivially valued case, $\bB_k^1$ is the union of closed intervals starting in the Gauss point $\eta_1$ and ending in the rig points of $\bB_k^1$. The topology is very simple as $\bB_k^1\setminus\{\eta_1\}$ is the topological disjoint union of the half-open intervals $(\eta_1,z]$ with $z$ ranging over the rig points of $\bB^1$. A subset $U$ of $\bB_k^1$ is an open neighborhood of $\eta_1$ if and only if $U \cap [\eta_1,z]$ is open for all rig points $z$ and equal to $[\eta_1,z]$ except for finitely many $z$. We refer to \cite[Example 1.4.4]{berkovic90:analytic_geometry} or \cite[\S 1.1.6]{boucksom2022globalpluripotentialtheorytrivially} for these facts, from which we deduce easily that the standard open subsets form also a basis of topology in the trivially valued case.
\end{art}

The main result of this appendix is the following lemma.

\begin{lem} \label{ball described by minimal polynomial}
Let $\bB^1(z,\rho)$ be a closed disc as above and let $f(T)$ be the minimal polynomial of the center $z$, i.e., a generator of the corresponding maximal ideal of $k\angles T$. Then there is a unique $r \in (0,1]$ such that
$$\bB^1(z,\rho)=\bigl\{x \in \bB_k^1\mid  |f(x)|\leq r\bigr\}.$$
\end{lem}

\begin{proof}
The argument is similar to  the Newton polygon technique from algebraic number theory. We write
$f(T)= (T-\alpha_1) \cdots (T-\alpha_d)$ for $\alpha_1, \dots , \alpha_d \in k^a$.
Let $k'$ be the completion of the algebraic closure $k^a$ as in~\artref{definition disks}.  We have $\alpha_1,\dots,\alpha_d\in k'^\circ$ by Gauss' lemma. Let $\pi\colon \bB_{k'}^1 \to \bB_k^1$ be the structure map, so we have
$$\pi^{-1}(\bB^1(z,\rho))= \bigl\{x'\in \bB_{k'}^1 \mid \min_{j=1,\dots,d} |(T-\alpha_j)(x')| \leq  \rho\bigr\}.$$
We set $\rho(x')\coloneqq  \min_{j=1,\dots,d} |(T-\alpha_j)(x')|$; this is crucial for the argument. Using that $\pi$ is surjective, it is enough to show that for any $x'\in \bB_{k'}^1$, the function $r(x')\coloneqq |f(x')|$ depends only on $\rho(x')$ and is strictly increasing in  $\rho(x')$. Then uniqueness follows  from the obvious fact $r(\bB_k^1)=[0,1]$.

We first prove that $r(x')$ depends only on $\rho(x')$ for $x'\in \bB_{k'}^1$. There is a root $\alpha$ of $f(T)$ in $k'$ such that
$$|(T-\alpha)(x')|=\min_{j=1,\dots,d} |(T-\alpha_j)(x')|=\rho(x');$$
we define
$$M \coloneqq \bigl\{j \in \{1,\dots,d\}\mid |(T-\alpha_j)(x')|= |(T-\alpha)(x')| \bigr\}.$$
Clearly, this does not depend on the choice of $\alpha$.
Since $\rho(x')$ and $r(x')$ are invariant under $\mathrm{Gal}(k^a/k)$, we may assume  that $\alpha=\alpha_1$.
For any $j\in\{1,\dots,d\}$, we have
\begin{equation} \label{two related inequalities}
|(T-\alpha)(x')| \leq |(T-\alpha_j)(x')| \quad \text{and} \quad |\alpha_j- \alpha| \leq  |(T-\alpha_j)(x')|
\end{equation}
and at most one of the two inequalities can be strict. Indeed, the first inequality is from the choice of $\alpha$. The second inequality and the claim about strictness follow from the ultrametric triangle inequality in $\sH(x')$. We deduce from \eqref{two related inequalities} that for $j \in M$, we have
$$|\alpha_j-\alpha| \leq  |(T-\alpha_j)(x')|= |(T-\alpha)(x')|=\rho(x')$$
and that for $j \in \{1,\dots,d\} \setminus M$, we have
$$|\alpha_j-\alpha| =  |(T-\alpha_j)(x')| > |(T-\alpha)(x')|=\rho(x').$$
Using $\alpha=\alpha_1$, this shows $M= \{j \in \{1,\dots,d\}\mid |\alpha_j-\alpha_1|\leq \rho(x')\}$, proving that $M$ depends only on $\rho(x')$. Let $m$ be cardinality of $M$. Then we have
$$r(x')=|f(x')|=|(T-\alpha_1)(x')| \cdots |(T-\alpha_d)(x')|= \rho(x')^m \cdot
\prod_{j \not\in M} |(\alpha_j-\alpha_1)(x')|,$$
which shows that $r(x')$ depends only on $\rho(x')$. To prove that $r(x')$ is strictly increasing in $\rho(x')$, let $x',y' \in \bB_{k'}^1$ with $\rho(y')>\rho(x')$. To prove $r(y')>r(x')$, we may again replace $x',y'$ by suitable conjugates to assume $\rho(x')=|(T-\alpha_1)(x')|$ and  $\rho(y')=|(T-\alpha_1)(y')|$. Using the above expression for $M$, we deduce that $M(x')\coloneqq \{j \in \{1,\dots,d\} \mid |(T-\alpha_j)(x')| =\rho(x')\}$ satisfies
$$M(x')=\bigl\{j \in \{1,\dots,d\} \mid |\alpha_j-\alpha_1| \leq \rho(x') \bigr\}.$$
Using the same for $y'$, we deduce $M(x') \subset M(y')$. It remains to prove the inequality in
$$r(x') = |(T-\alpha_1)(x')| \cdots |(T-\alpha_d)(x')| < r(y')= |(T-\alpha_1)(y')| \cdots |(T-\alpha_d)(y')|$$
which we will show by comparing the factors. For $j \in M(x')$, we use $$|(T-\alpha_j)(x')|=\rho(x')< \rho(y')=|(T-\alpha_j)(y')|$$ as $M(x') \subset M(y')$. For $j \in M(y') \setminus M(x')$, we have $|(T-\alpha_j)(x')|>|(T-\alpha_1)(x')|$, hence $$|(T-\alpha_j)(x')|= |\alpha_j-\alpha_1| \leq |(T-\alpha_j)(y')|$$ by \eqref{two related inequalities}.
For $j\in \{1,\dots,d\} \setminus M(y')$, the inclusion $M(x') \subset M(y')$ and \eqref{two related inequalities} show that
$$|(T-\alpha_j)(x')| = |\alpha_j-\alpha_1| = |(T-\alpha_j)(y')|.$$
These three displays prove $r(x') <r(y')$, showing the remaining claim of the lemma.
\end{proof}

\bibliographystyle{egabibstyle}
\bibliography{papers}

\end{document}